\documentclass[american,parskip=half]{scrartcl}
\usepackage{babel}
\usepackage[utf8]{inputenc}
\usepackage{amssymb,amsmath,amsthm,bm,mathtools}

\usepackage{graphicx,subcaption}
\usepackage{xcolor}
\usepackage{pgfplots}
\graphicspath{{matlab/}}

\usepackage{csquotes}
\usepackage[style=numeric, maxbibnames=9, giveninits=true%, backref=true
  ]{biblatex}

\usepackage[english,pdfusetitle,colorlinks]{hyperref}
\usepackage[english,compress]{cleveref}
\newtheorem{thm}{Theorem}[section]
\newtheorem{cor}[thm]{Corollary}
\newtheorem{lem}[thm]{Lemma}
\newtheorem{prop}[thm]{Proposition}

\theoremstyle{definition}
\newtheorem{defin}[thm]{Definition}
\Crefname{defin}{Definition}{Definitions}
\Crefname{rem}{Remark}{Remarks}
\Crefname{thm}{Theorem}{Theorems}
\Crefname{prop}{Proposition}{Propositions}
\crefname{defin}{def.}{def.}
\crefformat{defin}{def.~#2#1#3}
\Crefformat{defin}{Definition~#2#1#3}
\Crefname{lem}{Lemma}{Lemmas}
\newtheorem{rem}[thm]{Remark}

\newcommand\restr[2]{{% we make the whole thing an ordinary symbol
    \left.\kern-\nulldelimiterspace % automatically resize the bar with \right
    #1 % the function
    \vphantom{\big|} % pretend it's a little taller at normal size
    \right|_{#2} % this is the delimiter
}}

\usepackage{verbatim}

\newcommand{\N}{\ensuremath{\mathbb{N}}}
\newcommand{\NN}{\ensuremath{\mathbb{N}_0}}
\newcommand{\T}{\ensuremath{\mathbb{T}}}
\renewcommand{\S}{\ensuremath{\mathbb{S}}}

\newcommand{\Z}{\ensuremath{\mathbb{Z}}}

\newcommand{\R}{\ensuremath{\mathbb{R}}}
\newcommand{\C}{\ensuremath{\mathbb{C}}}

\newcommand{\abs}[1]{\ensuremath{\left\vert#1\right\vert}}
\newcommand{\inn}[1]{\ensuremath{\left\langle#1\right\rangle}}

\newcommand{\e}{\mathrm{e}}

\renewcommand{\i}{\mathrm{i}}
\newcommand{\zb}[1]{\ensuremath{\boldsymbol{#1}}}

\renewcommand{\d}{\, \mathrm{d}}

\newcommand{\norm}[1]{\left\lVert #1%\smash{#1} 
  \right\rVert}

\title{Approximation properties of the double Fourier sphere method}

\date{}
\author{%
  Sophie Mildenberger\thanks{TU Berlin, Institute of
    Mathematics, MA 4-3, Straße des 17. Juni 136, D-10623 Berlin, Germany.
  }\\%
  {\footnotesize\href{mailto:mildenberger@campus.tu-berlin.de}{mildenberger@campus.tu-berlin.de}}
  \and 
  Michael Quellmalz\footnotemark[1]\\%
  {\footnotesize\href{mailto:quellmalz@math.tu-berlin.de}{quellmalz@math.tu-berlin.de}}
}

\begin{document}

\maketitle

\begin{abstract}
  We investigate analytic properties of the double Fourier sphere (DFS) method, 
  which transforms a function defined on the two-dimensional sphere to a function defined on the two-dimensional torus.
  Then the resulting function can be written as a Fourier series yielding an approximation of the original function.
  We show that the DFS method preserves smoothness:
  it continuously maps spherical Hölder spaces into the respective spaces on the torus, 
  but it does not preserve spherical Sobolev spaces in the same manner. 
  Furthermore,
  we prove sufficient conditions for the absolute convergence of the resulting series expansion on the sphere as well as results on the speed of convergence.
  
  \medskip
  \noindent
  \textit{Math Subject Classifications.}
%  33C55,  % Spherical harmonics
  42B05,  % Fourier series and coefficients in several variables
  42C10,  % Fourier series in special orthogonal functions (Legendre polynomials, Walsh functions, etc.)
  %43A50,  % Convergence of Fourier series and of inverse transforms (Abstract harmonic analysis )
  43A90,  % Harmonic analysis and spherical functions
  65T50  % Numerical methods for discrete and fast Fourier transforms
\end{abstract}

\section{Introduction}

The problem of approximating functions defined on the two-dimensional sphere arises in many real-world applications such as weather prediction and climate modeling \cite{w1,w2,w3}. 
One approach is the grid-point method, in which functions and equations are approximated on a finite set of points \cite{grids}. 
For example, the German weather service's short forecast model uses a rotated latitude-longitude grid \cite{cosmo}. 
Another common approach to approximate spherical functions is the spectral method \cite{spectral}, where a spherical function is expanded into a series of basis functions. 
This method is especially relevant to applications involving differential equations, since 
derivatives can be evaluated exactly via the derivatives of the basis functions. 

A frequently used choice 
of basis functions for the spectral method are the {spherical harmonics} \cite{michel}. 
They are eigenfunctions of the spherical Laplacian, 
which makes them a natural choice for problems such as
solving differential equations \cite{yee}, 
spherical deconvolution \cite{HiQu15}, 
the approximation of measures \cite{EhlGraNeuSte21},
or
modeling the earth's upper mantle \cite{HiPoQu18,WoDz84} and gravitational field \cite{Ger14}.
While the naive computation of spherical harmonics expansions is very memory and time consuming \cite{yee}, 
there are methods allowing for a faster computation, known as fast spherical Fourier transforms, see, e.g., \cite{drhe,kunis,Mo99,WedHamMoz13} and \cite[sec.~9.6]{numFou}.
However, these algorithms do not reach the performance of fast Fourier transforms on the torus and they suffer from  difficulties in the numerical evaluation of associated Legendre functions, cf.\ \cite{Sch13}.

The \emph{double Fourier sphere} (DFS) method avoids numerical difficulties of spherical harmonics expansions.
The core concept is to apply a simple transformation, which is closely related to the spherical {longitude-latitude coordinates}, before any approximation steps. 
A spherical function is transformed to a biperiodic function on a rectangular domain, i.e., a function on the two-dimensional torus.
The resulting function can, in turn, be represented via a two-dimensional Fourier series, which can be efficiently approximated by a {fast Fourier transform} (FFT), cf.\ \cite{yee} and \cite[ch.~5,~7]{numFou}.

The DFS method was first proposed in 1972 by Merilees \cite{merilees} in the context of shallow water equations.
Further applications followed in meteorology \cite{boyd,orszag,yee} and geophysics \cite{fornberg}. 
In 2016, Townsend et al.\ proposed a DFS method 
for a low-rank approximation of spherical functions \cite{townsend}. 
Sampling methods on the sphere based on spherical harmonics and the DFS method were compared in \cite{potts}. 
Recently, the DFS method was utilized in the numerical solution of partial differential equations on the sphere \cite{MonNak18} and the ball \cite{BouTow20,BouSloTow21}
as well as in computational spherical harmonics analysis \cite{DraWri20}.
To the best of our knowledge, no results on the convergence of the approximation with the DFS method were published so far.

In this paper, we are concerned with analytic properties of the DFS method.
The main contributions of this work are the following: 
\begin{enumerate}
\item[(i)]
We examine which function spaces are preserved under the DFS transformation. 
In particular, we show in \Cref{thm:diffC,thm:hoelC} that the DFS method maps the differentiability space $\mathcal{C}^k(\mathbb{S}^2 )$ and the Hölder space $\mathcal{C}^{k,\alpha}(\mathbb{S}^2)$ on the sphere~$\S^2$ continuously into the respective spaces $\mathcal{C}^k(\mathbb{T}^2)$ and $\mathcal{C}^{k,\alpha}(\mathbb{T}^2)$ on the torus~$\T^2$. 
However, in \Cref{thm:count} we prove that the analogue does not hold true for spherical Sobolev spaces.
\item[(ii)]
The Fourier series that results from applying the DFS method admits a certain symmetry on the torus. 
This allows us to provide a series expansion of the spherical function in terms of basis functions that are orthogonal with respect to a weight on the sphere in \Cref{thm:F}.
We prove the absolute convergence as well as convergence rates for this expansion of Hölder continuous functions in \Cref{thm:hoelScon,thm:hoelSrate}.
\item[(iii)]
Numerical tests indicate that this series expansion provides a comparable approximation quality while being faster than a spherical harmonics expansion.
\end{enumerate}

This paper is structured as follows:
\Cref{sec:DFS} introduces the DFS method.
In \Cref{sec:hoel}, we define Hölder spaces and related function spaces on the sphere and on the torus.
We show that the DFS method preserves differentiability and Hölder spaces in \Cref{chap:hoelDFS}.
\Cref{chap:3} is dedicated to the approximation of DFS functions via the Fourier series that results from applying the DFS method. 
We study sufficient conditions for the absolute convergence of this series and present bounds on its speed of convergence. 
Numerical results are presented in \Cref{sec:numerics}.
Finally, \Cref{sec:sobolev} addresses the question whether the space preserving properties of the DFS method on Hölder spaces also hold true for Sobolev spaces. 
This question is answered negative for the Sobolev space $H^1(\S^2)$.

\section{The double Fourier sphere method}
\label{sec:DFS}

The {double Fourier sphere (DFS) method} 
transforms spherical functions into functions on the two-dimensional torus.
Let $d \in \mathbb{N}$.
We denote the {$d$-dimensional torus} by
$
    {\mathbb{T}^d := \mathbb{R}^d / (2\pi \mathbb{Z}^d)}
$
and the {$d$-dimensional unit sphere} by
\begin{equation*}
    \mathbb{S}^d:=\{\bm{x} \in \mathbb{R}^{d+1}\mid \lVert \bm{x} \rVert=1\},
\end{equation*}
where $\norm{\cdot}$ denotes the Euclidean norm.
Note that every function defined on $\mathbb{T}^d$ can be identified with a function defined on $\mathbb{R}^d$ which is $2\pi$-periodic in all dimensions. 

Then the \emph{DFS coordinate transform} given by
\begin{equation}\label{eq:DFScoord}
  \phi \colon \mathbb{T}^2 \to \mathbb{S}^2, \ \phi(\lambda, \theta) := (\cos \lambda\, \sin \theta, \sin \lambda\, \sin \theta, \cos \theta),
\end{equation}
is well-defined on $\mathbb{T}^2$, since trigonometric functions are $ 2 \pi $-periodic.
The restriction of $\phi$ to $[-\pi,\pi)\times (0,\pi) \cup \{(0,0),(0,\pi)\}$
is bijective and known as the \emph{longitude-latitude transform} of spherical coordinates. 
Any spherical function can be composed with the longitude-latitude transform resulting in a function defined on the rectangle $[-\pi,\pi] \times [0,\pi]$, which is $2\pi$-periodic in latitude direction $\lambda$, but generally not periodic in $\theta$-direction. 
The longitude-latitude transform is illustrated in \Cref{fig:LLtransform}.

\begin{figure}[ht]
    \begin{subfigure}{.3\textwidth}
        \flushright
	    \includegraphics[width=0.6\textwidth]{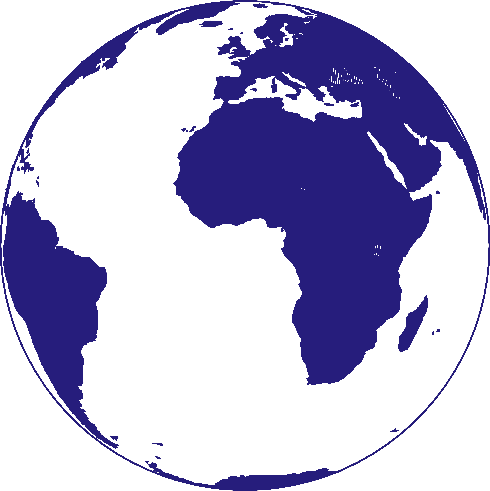}
    \end{subfigure}
    \hspace{9pt}
    {$\xrightarrow[\text{transform}]{\text{longitude-latitude}}$} \hspace{9pt}
    \begin{subfigure}{.375\textwidth}
	    \includegraphics[width=\textwidth]{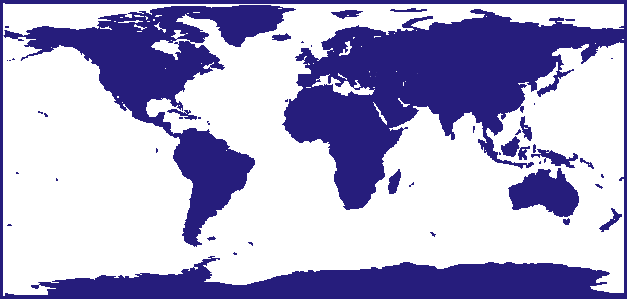}
    \end{subfigure}
	\caption{The longitude-latitude transform applied to the earth's land masses.}
	\label{fig:LLtransform}
\end{figure}

We want to approximate a spherical function $f\colon \mathbb{S}^2 \to \mathbb{C}$ using Fourier analysis on the two-dimensional torus. 
To this end, we require a transformation which yields functions defined on the torus, or equivalently $2\pi$-biperiodic functions with domain $[-\pi,\pi]^2$.
We call 
\begin{equation}\label{eq:DFS}
  \tilde{f} \colon \mathbb{T}^2 \to \mathbb{C}, \quad \tilde{f}:=f\circ \phi
\end{equation} the \emph{DFS function} of $f$.

Applying fundamental trigonometric identities, we notice that for all $(\lambda,\theta) \in \T^2$,
\begin{equation} \label{eq:BMC}
  \begin{alignedat}{5} 
    & \phi(\lambda + \pi,-\theta) && =(\cos(\lambda + \pi)\sin(-\theta),&& \ \sin(\lambda + \pi)\sin(-\theta),&& \ \cos(-\theta)) && \\
    & && =((-\cos \lambda)(-\sin\theta ),&& \ (-\sin\lambda)(-\sin \theta),&& \ \cos(\theta) ) && =\phi(\lambda,\theta) .
  \end{alignedat}
\end{equation}

The definition of the DFS function is often given in the equivalent form
\begin{equation}\label{eq:DFS1}
    \tilde{f}\colon [-\pi,\pi]^2 \to \mathbb{C}, \ \tilde{f}(\lambda, \theta) := \left\{\begin{alignedat}{2} 
        & g(\lambda+\pi,\theta), &&\ (\lambda,\theta) \in [-\pi,0]\times [0,\pi],\\
        & h(\lambda,\theta), && \ (\lambda,\theta)\in [0,\pi] \times [0,\pi],\\
        & g(\lambda, -\theta), && \ (\lambda,\theta) \in [0,\pi]\times[-\pi,0],\\
        & h(\lambda+\pi,-\theta),  && \ (\lambda,\theta) \in [-\pi,0]\times [-\pi,0],
    \end{alignedat}\right.
\end{equation}
where $h(\lambda,\theta)=f \circ {\phi}(\lambda,\theta)$ and $g(\lambda,\theta)=f \circ {\phi}(\lambda-\pi,\theta)$ for all $(\lambda,\theta) \in [0,\pi]^2$,
cf.\ \cite{townsend}.
\Cref{eq:DFS1} reveals the geometric interpretation of DFS functions. 
Applying the DFS method to a function $f$ corresponds to first composing $f$ with the longitude-latitude transform and then combining the resulting transformation with its \emph{glide reflection}, i.e., the domain of the transformed function is doubled by reflecting it over the $\lambda$-axis and subsequently translating the reflection by $-\pi$ in $\lambda$-direction; the translation is understood as acting on the one-dimensional torus. 
\Cref{fig:DFS} shows the application of the DFS method.

\begin{figure}
    \begin{subfigure}{.4\textwidth}
        \centering
        \includegraphics[width=.9\textwidth]{graphics/sphere.png}
    \end{subfigure}
    {\Large$\xrightarrow{\text{DFS method}}$}
    \begin{subfigure}{.4\textwidth}
        \centering
        \includegraphics[width=.9\textwidth]{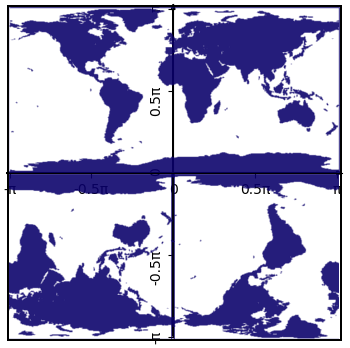}
    \end{subfigure}
    \caption{The DFS method applied to the landmasses of the earth.}
    \label{fig:DFS}
\end{figure}

A function $g\colon\T^2\to\C$ that satisfies 
\begin{equation} \label{eq:BMC2}
  g(\lambda,\theta) = g(\lambda+\pi,-\theta)
  ,\qquad (\lambda,\theta)\in\T^2,
\end{equation}
is called a \emph{block-mirror-centrosymmetric (BMC) function}.
If $g$ is also constant along the lines $\theta=0$ and $\theta=\pi$,
it is called a \emph{type-1 block-mirror-centrosymmetric (BMC-1) function}, see \cite[def.~2.2]{townsend}.
For any spherical function $f\colon\S^2\to\C$,
its DFS function $\tilde f$
is a BMC-1 function, 
which follows from the symmetry relation \eqref{eq:BMC}
and the fact that the lines $\theta=0$ and $\theta=\pi$ correspond to the north and south pole of the sphere, respectively.
Conversely, for
any BMC-1 function $g\colon\T^2\to\C$,
there exists a spherical function $f\colon \S^2\to\C$ such that $g$ is the DFS function of $f$.

\section{Hölder spaces on the sphere and on the torus}
\label{sec:hoel}

Since the notation in the literature varies slightly, we clarify the notion of different function spaces we use in the following. In the definitions of differentiability spaces and Hölder spaces on open subsets $U$ of $\mathbb{R}^d$, we follow \cite[p.~2~ff.]{triebel}.
Let $d\in\mathbb{N}$
and $k \in \NN$. 
We set
\begin{equation*}
    B_k^d:=\{\bm{\beta} \in \NN^d\mid \lvert \bm{\beta}\rvert \leq k\},
\end{equation*} 
where 
$
  \lvert \bm{\beta} \rvert := \sum_{i=1}^d \beta_i
$
denotes the $1$-norm of a multiindex $\bm{\beta} \in \NN^d$.

\begin{defin}[Spaces of differentiable functions]\label{defin:cont}
    Let $U \subset \mathbb{R}^d$ and $k \in \NN$. We define
    \begin{equation*}
        \mathcal{C}(U,\C^m)
        := \mathcal{C}^0(U,\C^m):=\{f \colon U \to \C^m \mid f \text{ is bounded and uniformly continuous}\}.
    \end{equation*}
    For $k\ge1$, we additionally assume that $U$ is open and define
    \begin{equation*}
        \mathcal{C}^k(U,\C^m)
        :=\{f \in \mathcal{C}(U,\C^m)\mid D^{\bm{\beta}}f \in \mathcal{C}(U,\C^m) \text{ for all } \bm{\beta} \in B^d_k\}.
    \end{equation*}
    Equipping this space
    with the norm
    \begin{equation*}
        \lVert f \rVert_{\mathcal{C}^k(U,\C^m)}
        :=\sum_{\bm{\beta} \in B^d_k} \sup_{\bm{x} \in U} \lVert D^{\beta} f(\bm{x})\rVert,
    \end{equation*}
    it becomes a Banach space.
\end{defin}

\begin{defin}[Lipschitz spaces]\label{defin:lip}
    Let 
    $U\subset \mathbb{R}^d$. 
    For $f\in \mathcal{C}(U,\C^m)$, we define the \emph{Lipschitz seminorm} by
    \begin{equation*}
      \lvert f \rvert_{\mathrm{Lip}(U,\C^m)} 
      := \sup_{{\bm{x}, \bm{y} \in U, \, \bm{x} \not= \bm{y}}} \frac{\lVert f(\bm{x})-f(\bm{y}) \rVert}{\lVert \bm{x}-\bm{y}\rVert}.
    \end{equation*}
    The \emph{Lipschitz space}
    \begin{equation*}
        \mathrm{Lip}(U,\C^m):=\{f\in \mathcal{C}(U,\C^m)\mid \lvert f \rvert_{\mathrm{Lip}(U,\C^m)}<\infty\},
    \end{equation*}
    is a Banach space equipped with the norm $\lVert \cdot \rVert_{\mathrm{Lip}(U,\C^m)} := \lVert \cdot \rVert_{\mathcal{C}(U,\C^m)}+\lvert \cdot \rvert_{\mathrm{Lip}(U,\C^m)}$. 
\end{defin}

\begin{defin}[Hölder spaces]\label{defin:hoel}
    Let $k \in \NN,\ 0<\alpha<1$, and let $U\subset \mathbb{R}^d$ be open. For $f \in \mathcal{C}^k(U)$, we define the $\mathcal{C}^{k,\alpha}$-\emph{seminorm}
    \begin{equation*}
        \lvert f \rvert_{\mathcal{C}^{k,\alpha}(U,\C^m)}
        :=\sup_{\substack{\bm{x},\bm{y} \in U,\,\bm{x}\not=\bm{y} \\\bm{\beta}\in B^d_k,\,\lvert\bm{\beta}\rvert=k}} \frac{\lVert D^{\bm{\beta}}f(\bm{x})-D^{\bm{\beta}}(\bm{y})\rVert}{\lVert \bm{x}-\bm{y}\rVert^\alpha}.
    \end{equation*}
    The $(k,\alpha)$-\emph{Hölder space} 
    \begin{equation*}
        \mathcal{C}^{k,\alpha}(U,\C^m):=\{f \in \mathcal{C}^k(U,\C^m) \mid \lvert f \rvert_{\mathcal{C}^{k,\alpha}(U)}<\infty\}
    \end{equation*}
    is a Banach space equipped with the norm $\lVert \cdot \rVert_{\mathcal{C}^{k,\alpha}(U,\C^m)}:=\lVert \cdot \rVert_{\mathcal{C}^k(U,\C^m)}+\lvert \cdot \rvert_{\mathcal{C}^{k,\alpha}(U,\C^m)}$.
    For ease of notation we interchangeably use $\mathcal{C}^k=\mathcal{C}^{k,0}$.
\end{defin}

For spaces of scalar-valued functions, we write $\mathcal{C}^{k,\alpha}(U) = \mathcal{C}^{k,\alpha}(U,\C)$.
The notions of Lipschitz and Hölder spaces transfer to the torus.

\begin{defin}[Function spaces on the torus] \label{defin:torus}
For $k \in \NN$ and $0\le\alpha<1$,
we define the \emph{Lipschitz space}
$
  \mathrm{Lip}(\mathbb{T}^d)
$
and the $(k,\alpha)$-\emph{Hölder space} 
$
  \mathcal{C}^{k,\alpha}(\mathbb{T}^d)
$
on the $d$-dimensional torus $\mathbb{T}^d$
as the restriction of $\mathrm{Lip}(\mathbb{R}^d)$ and $\mathcal{C}^{k,\alpha}(\mathbb{R}^d)$, respectively,
to functions that are $2\pi$-periodic in all dimensions.
\end{defin}

\begin{defin}[Hölder spaces on the sphere]\label{defin:sphere}
    Let $k \in \NN$, $0\leq \alpha <1$,
    and $f \colon\mathbb{S}^d \to \C$.
    If there exists $f^\ast \in \mathcal{C}^{k,\alpha}(U)$ on an open set $U \supset  \mathbb{S}^d$ such that
    $\restr{f^\ast}{\mathbb{S}^d}=f$,
    we call $f^\ast$ a $\mathcal{C}^{k,\alpha}$-\emph{extension} of $f$.
    For $\alpha=0$, we define the $\mathcal{C}^k$-\emph{extension seminorm} by
    \begin{equation}\label{eq:Ck-S}
        \big \lvert f^* \big \rvert_{\mathcal{C}^k(\mathbb{S}^d)}^*:=\sum_{\bm{\beta}\in B^{d+1}_k}  \big \lVert D^{\bm{\beta}}f^* \big \rVert_{\mathcal{C}(\mathbb{S}^d)}.
    \end{equation}
    For $\alpha>0$ and a $\mathcal{C}^{k,\alpha}$-extension $f^*$, we define the $\mathcal{C}^{k,\alpha}$-\emph{extension seminorm} by
    \begin{equation*}
        \big \lvert f^* \big \rvert_{\mathcal{C}^{k,\alpha}(\mathbb{S}^d)}^* :=\big\lvert f^* \big \rvert_{\mathcal{C}^k(\mathbb{S}^d)}^*+ \sup_{\substack{\bm{\xi},\bm{\eta}\in \mathbb{S}^d, \ \bm{\xi}\not=\bm{\eta}\\ \bm{\beta}\in B^{d+1}_k, \ \lvert \bm{\beta}\rvert=k}} \frac{\big\lvert D^{\bm{\beta}}f^*(\bm{\xi})-D^{\bm{\beta}}f^*(\bm{\eta})\big\rvert}{\lVert \bm{\xi}-\bm{\eta} \rVert^\alpha}.
    \end{equation*}
    We define the $C^{k,\alpha}$-\emph{norm} of $f \colon\mathbb{S}^d \to \C$ by
    \begin{equation*}
        \lVert f \rVert_{\mathcal{C}^{k,\alpha}(\mathbb{S}^d)}:=\inf \left\{ \big \lvert f^* \big\rvert_{\mathcal{C}^{k,\alpha}(\mathbb{S}^d)}^*\mid f^* \text{ is a } \mathcal{C}^{k,\alpha} \text{-extension of } f \right\}.
    \end{equation*}
    The $(k,\alpha)$-\emph{Hölder space on the $d$-dimensional sphere}, defined by
    \begin{equation*}
        \mathcal{C}^{k,\alpha}(\mathbb{S}^d):=\{f\colon\mathbb{S}^d \to \C \mid \lVert f \rVert_{\mathcal{C}^{k,\alpha}(\mathbb{S}^d)}<\infty\},
    \end{equation*}
    is a Banach space equipped with this norm.
\end{defin}

If $\alpha>0$ and $f \in \mathcal{C}^{k,\alpha}(X)$ by any of the \Cref{defin:hoel,defin:sphere}, we say that $f$ is $(k,\alpha)$-Hölder continuous on $X$, or simply $\alpha$-Hölder continuous if $k=0$. 
If in one of the \Cref{defin:cont,defin:lip,defin:hoel,defin:torus,defin:sphere}
we have $k=0$, we omit $k$.
The following result shows that Lipschitz continuity implies Hölder continuity.

\begin{prop}\label{prop:lipH}
	Let $U \subset \mathbb{R}^d$ be an open set, $V \subset U$, and $f:U \to \mathbb{C}^m$. If $f$ is bounded and Lipschitz-continuous on $V$, then $f$ is $\alpha$-Hölder continuous on $V$ for all $0<\alpha<1$ with 
	\begin{equation}\label{eq:lipH1}
		\lvert f \rvert_{\mathcal{C}^{\alpha}(V,\C^m)}\leq \max\{\lvert f \rvert_{\mathrm{Lip}(V,\C^m)},2\lVert f\rVert_{\mathcal{C}(V,\C^m)}\}.
	\end{equation} 
	Furthermore, if $V$ is convex
	and $f\in\mathcal{C}^1(U)$, then $f$ is Lipschitz-continuous on $V$ with
	\begin{equation}\label{eq:lipH2}
		\lvert f \rvert_{\mathrm{Lip}(V)}
    \leq \lVert \nabla f \rVert_{\mathcal{C}(U,\R^d)}
    \leq \lVert f \rVert_{\mathcal{C}^1(U)},
	\end{equation}
  where $\nabla f=(D^i f)_{i=1}^d$.
\end{prop}

\begin{proof}
	\begin{comment}
	It is easy to see that continuously differentiable functions are locally Lipschitz-continuous, this directly implies the Lipschitz-continuity on any compact set. 
	\end{comment}
	Let $f$ be bounded and Lipschitz-continuous on $V$ and let $\bm{x},\bm{y} \in V$ with $\bm{x} \not=\bm{y}$. Then
	\begin{equation*}
		\frac{\lVert f(\bm{x})-f(\bm{y})\rVert}{\lVert \bm{x}-\bm{y}\lVert^\alpha}\leq 
		\begin{cases}
			\frac{\lVert f(\bm{x})-f(\bm{y})\Vert}{\lVert \bm{x}-\bm{y}\rVert}\leq \lvert f \rvert_{\mathrm{Lip}(V)}, & \lVert \bm{x}-\bm{y}\rVert \leq 1,\\
			\lVert f(\bm{x})-f(\bm{y})\rVert \leq 2 \lVert f\rVert_{\mathcal{C}(V)}, & \lVert \bm{x}-\bm{y}\rVert\geq 1,
		\end{cases}
	\end{equation*}
	which proves \eqref{eq:lipH1}.
  For $f\in\mathcal{C}^1(U)$, applying the mean value theorem  to $${[0,1]\ni t\mapsto f(\bm{x}+t(\bm{y}-\bm{x}))}$$ shows that there exists some $\tau\in [0,1]$ such that
  \begin{equation*}
    f(\bm{y})-f(\bm{x})
    = \inn{ \nabla f(\bm{x}+\tau(\bm{y}-\bm{x})), \bm{y}-\bm{x} }.
  \end{equation*}
  Taking the absolute value and the Cauchy--Schwarz inequality yield
  \begin{equation*}
    \lvert f(\bm{y})-f(\bm{x}) \rvert
    \le \norm{ \nabla f(\bm{x}+t(\bm{y}-\bm{x}))} \norm{ \bm{y}-\bm{x} }
    \le \norm{ \nabla f}_{\mathcal{C}(U,\R^d)} \norm{ \bm{y}-\bm{x} }.\qedhere
  \end{equation*}
\end{proof}

\section{Hölder continuity of DFS functions}\label{chap:hoelDFS}

In this section, we investigate sufficient conditions on a spherical function to ensure Hölder continuity of its DFS function. 
In particular, 
we show that the DFS function of a function in $\mathcal{C}^{k+1}(\mathbb{S}^2)$ or in $\mathcal{C}^{k,\alpha}(\mathbb{S}^2)$ is in the Hölder space $\mathcal{C}^{k,\alpha}(\mathbb{T}^2)$ 
and we prove upper bounds on the Hölder norms of such DFS functions.
This will allow us to obtain approximation properties of its Fourier series in \Cref{chap:3}.

We present two lemmas, which are necessary to prove the main results of this section.
The following technical lemma shows explicit bounds on the number of summands in the multivariate chain rule for higher derivatives of vector-valued functions.
We use the abbreviation $[n] \coloneqq \{1,2,\dots,n\}$ for $n\in\mathbb{N}$.

\begin{lem}\label{lem:faaDi}
    Let $U \subset \mathbb{R}^2$ and $V \subset \mathbb{R}^3$ be open sets, and both $h=(h_1,h_2,h_3) \colon U \to V$ and $g \colon V \to \mathbb{C}$ be at least $k$ times continuously differentiable for some $k \in \mathbb{N} $. Then, for any $\bm{\beta} \in B^2_k$, we have
    \begin{equation}\label{eq:indH6}
        D^{\bm{\beta}} (g \circ h)=  \sum_{i=1}^{n_{\bm{\beta}}} (D^{\bm{\gamma}_{\bm{\beta},i}}g)\circ h \cdot  \prod_{j=1}^{m_{\bm{\beta},i}} D^{\bm{\mu}_{\bm{\beta},i,j}}h_{l_{\bm{\beta},i,j}},
    \end{equation}
    where the constants fulfill
    \begin{align}
      n_{\bm{\beta}} \in \NN, \ &n_{\bm{\beta}}\leq 2^{-1} \, (k+2)!,\label{eq:indH1} \\
      \bm{\gamma}_{\bm{\beta},i} \in B^3_k, \ &i \in [{n_{\bm{\beta}}}],\label{eq:indH2} \\
      m_{\bm{\beta},i} \in B^1_k, \ &i \in [{n_{\bm{\beta}}}],\label{eq:indH3}\\
      \bm{\mu}_{\bm{\beta},i,j} \in B^2_k, \ &i\in [{n_{\bm{\beta}}}], \ j \in [m_{\bm{\beta},i}],\label{eq:indH4}\\
      l_{\bm{\beta},i,j}\in [3], \ &i \in [{n_{\bm{\beta}}}], \ j \in [m_{\bm{\beta},i}].\label{eq:indH5}
    \end{align}
\end{lem}

\begin{proof} 
    We prove the lemma by induction over $k$.
    
    \underline{Base case $k=1$:} Let $h$ and $g$ fulfill the conditions of the lemma for $k=1$. Clearly, the claim holds for $\bm{\beta}=0$. For $\bm{\beta}\in B^2_1\setminus \{0\}$, we can apply the multi-dimensional chain rule since $g$ and $h$ are continuously differentiable and we get
    \begin{equation*}
        D^{\bm{\beta}}(g \circ h)
        =\sum_{p=1}^3 \big( D^{\bm{e}^p} g \big) \circ h \cdot D^{\bm{\beta}}h_p,
    \end{equation*}
    where $\bm{e}^p$ denotes the $p$-th unit vector.
    The statement holds with $n_{\bm{\beta}}=3 = 2^{-1} (k+2)!$.
    
    \underline{Induction step:}
    Let $h$ and $g$ fulfill the conditions of the lemma for $k>1$ and assume the lemma holds for $k-1$. 
    For all $\bm{\beta} \in B^2_{k-1}$, the statement holds since we can replace $k-1$ by $k$ in the bounds on the constants. 
    Let $\bm{\beta}^+ \in B^2_k \setminus B^2_{k-1}$, then there exists $p \in \{1,2\}$ such that $\bm{\beta^{+}} = \bm{e}^p+ \bm{\beta}$ for some $\bm{\beta} \in B^2_{k-1}$. 
    By the induction hypothesis, we can choose constants which satisfy \Cref{eq:indH6,eq:indH1,eq:indH2,eq:indH3,eq:indH4,eq:indH5} for $\bm{\beta}$ and $k-1$. Since $g$ and $h$ are $k$ times continuously differentiable, their composition is also $k$ times continuously differentiable and we can apply \eqref{eq:indH6}:
    \begin{equation*}
        D^{\bm{\beta^{+}}}(g \circ h) = D^{\bm{e}^p+\bm{\beta}}(g \circ h)= D^{\bm{e}^p} \Big(D^{\bm{\beta}} (g \circ h) \Big )
        =D^{\bm{e}^p} \Big ( \sum_{i=1}^{n_{\bm{\beta}}} (D^{\bm{\gamma}_{\bm{\beta},i}}g)\circ h \cdot  \prod_{j=1}^{m_{\bm{\beta},i}} D^{\bm{\mu}_{\bm{\beta},i,j}}h_{l_{\bm{\beta},i,j}}\Big).
    \end{equation*}
    The product rule yields
    \begin{equation*}
      \begin{alignedat}{2}
        &  && D^{\bm{\beta^{+}}}(g \circ h) \\
        & = && \underbrace{\sum_{i=1}^{n_{\bm{\beta}}} D^{\bm{e}^p}\Big((D^{\bm{\gamma}_{\bm{\beta},i}}g)\circ h\Big) \cdot  \prod_{j=1}^{m_{\bm{\beta},i}} D^{\bm{\mu}_{\bm{\beta},i,j}}h_{l_{\bm{\beta},i,j}}}_{=:A} +\underbrace{\sum_{i=1}^{n_{\bm{\beta}}} (D^{\bm{\gamma}_{\bm{\beta},i}}g)\circ h \cdot  D^{\bm{e}^p} \Big(\prod_{j=1}^{m_{\bm{\beta},i}} D^{\bm{\mu}_{\bm{\beta},i,j}}h_{l_{\bm{\beta},i,j}}\Big)}_{=:B}.
      \end{alignedat}
    \end{equation*}
    By assumption $g$ is $k$ times continuously differentiable and $\lvert \bm{\gamma_{\beta,i}} \rvert \leq k-1$ by \eqref{eq:indH2}, therefore $D^{\bm{\gamma}_{\bm{\beta},i}}g$ is continuously differentiable for all $i$. Since $h$ is also continuously differentiable, we can apply the multi-dimensional chain rule to the summands in
    \begin{equation*}
        \begin{alignedat}{3}
            & A \ && =&&\sum_{i=1}^{n_{\bm{\beta}}} D^{\bm{e}^p}\Big((D^{\bm{\gamma}_{\bm{\beta},i}}g)\circ h\Big) \cdot  \prod_{j=1}^{m_{\bm{\beta},i}} D^{\bm{\mu}_{\bm{\beta},i,j}}h_{l_{\bm{\beta},i,j}}\\
            & && = \ && \sum_{i=1}^{n_{\bm{\beta}}} \sum_{q=1}^3 (D^{(\bm{e}^q+\bm{\gamma}_{\bm{\beta},i})}g)\circ h \cdot D^{\bm{e}^p}h_q \cdot \prod_{j=1}^{m_{\bm{\beta},i}} D^{\bm{\mu}_{\bm{\beta},i,j}}h_{l_{\bm{\beta},i,j}}.
        \end{alignedat}
    \end{equation*}
    Similarly, since $h$ is $k$ times continuously differentiable and $\lvert\bm{\mu}_{\bm{\beta},i,j}\rvert\leq k-1$ by \eqref{eq:indH5}, we know that $D^{\bm{\mu}_{\bm{\beta},i,j}}h_{l_{\bm{\beta},i,j}}$ is continuously differentiable for all $i$ and all $j$ and we can apply the product rule to 
    \begin{equation*}
        \begin{aligned}
            B &  = \sum_{i=1}^{n_{\bm{\beta}}} (D^{\bm{\gamma}_{\bm{\beta},i}}g)\circ h \cdot  D^{\bm{e}^p} \Big(\prod_{j=1}^{m_{\bm{\beta},i}} D^{\bm{\mu}_{\bm{\beta},i,j}}h_{l_{\bm{\beta},i,j}}\Big)\\
            &= \sum_{i=1}^{n_{\bm{\beta}}} \sum_{r=1}^{m_{\bm{\beta},i}} (D^{\bm{\gamma}_{\bm{\beta},i}}g)\circ h \cdot D^{\bm{e}^p+\bm{\mu}_{\bm{\beta},i,r}}h_{l_{\bm{\beta},i,r}} \cdot \prod_{\substack{j=1\\ j\not=r}}^{m_{\bm{\beta},i}} D^{\bm{\mu}_{\bm{\beta},i,j}}h_{l_{\bm{\beta},i,j}}.
        \end{aligned}
    \end{equation*}
    We combine the resulting sums and relabel the constants to obtain 
    \begin{equation*}
        D^{\bm{\beta}^+} (g \circ h)=  \sum_{i=1}^{n_{\bm{\beta}^+}} (D^{\bm{\gamma}_{\bm{\beta}^+,i}}g)\circ h \cdot  \prod_{j=1}^{m_{\bm{\beta}^+,i}} D^{\bm{\mu}_{\bm{\beta}^+,i,j}}h_{l_{\bm{\beta}^+,i,j}},
    \end{equation*}
    where for all  $i_1 \in [n_{\bm{\beta}^+}]$ and $j_1 \in [m_{\bm{\beta}^+,i_1}]$ there exist $i_2 \in [n_{\bm{\beta}}]$ and $j_2 \in [m_{\bm{\beta},i_2}]$ such that the constants satisfy 
    \allowdisplaybreaks%
    \begin{align*}
            &\bullet \ \bm{\gamma}_{\bm{\beta}^+,i_1}=\bm{e}^{l_{\bm{\beta}^+,i_1,j_1}}+\bm{\gamma}_{\bm{\beta},i_2} \text{ or } \bm{\gamma}_{\bm{\beta}^+,i_1}=\bm{\gamma}_{\bm{\beta},i_2} && \underset{\mathrm{\eqref{eq:indH2}}}{\implies} \lvert \bm{\gamma}_{\bm{\beta}^+,i_1}\rvert  \leq k,\\
            & \bullet \ m_{\bm{\beta}^+,i_1}=1+m_{\bm{\beta},i_2} \text{ or }  m_{\bm{\beta}^+,i_1}=m_{\bm{\beta},i_2} && \underset{\mathrm{\eqref{eq:indH3}}}{\implies}  m_{\bm{\beta}^+,i_1}\leq k,\\
            & \bullet \ \bm{\mu}_{\bm{\beta}^+,i_1,j_1}=\bm{e}^p, \ \bm{\mu}_{\bm{\beta}^+,i_1,j_1}= \bm{\mu}_{\bm{\beta},i_2,j_2} \text{, or } \bm{\mu}_{\bm{\beta}^+,i_1,j_1}=\bm{e}^p+\bm{\mu}_{\bm{\beta},i_2,j_2} && \underset{\mathrm{\eqref{eq:indH4}}}{\implies} \lvert \bm{\mu}_{\bm{\beta}^+,i_1,j_1} \rvert \leq k,\\
            & \bullet \ l_{\bm{\beta}^+,i_1,j_1} \in [3],
    \end{align*}%
    \allowdisplaybreaks[0]%
    i.e., the constants satisfy \Cref{eq:indH6,eq:indH2,eq:indH3,eq:indH4,eq:indH5} for $\bm{\beta}^+$ and $k$. It remains to prove \eqref{eq:indH1} for $n_{\bm{\beta}^+}$:
    The sum resulting from $A$ has $3{n_{\bm{\beta}}}$ summands and the sum resulting from $B$ has $\sum_{i=1}^{n_{\bm{\beta}}} m_{\bm{\beta},i}\leq \max \{m_{\bm{\beta},i}, \ i\in [{n_{\bm{\beta}}}]\} \, {n_{\bm{\beta}}} \leq (k-1)\, {n_{\bm{\beta}}}$ summands. Hence, by \eqref{eq:indH1}, we obtain 
    \begin{equation*}
         n_{\bm{\beta}^+} \leq (3+k-1) n_{\bm{\beta}} \leq (k+2) \, 2^{-1} (k-1+2)!= 2^{-1} (k+2)!. \qedhere
    \end{equation*}
\end{proof}

\begin{lem}\label{lem:DFSBd}
    The DFS coordinate transform $\phi \colon \mathbb{R}^2 \to \mathbb{S}^2$ is infinitely differentiable with Lipschitz-continuous derivatives and $(k,\alpha)$-Hölder continuous for all $k \in \NN$ and $0<\alpha<1$. For all $\bm{\mu}\in \NN^2$ and $l \in [3]$, we have
    \begin{align}
        \big\lVert D^{\bm{\mu}} \phi_l \big\rVert_{\mathcal{C}(\mathbb{R}^2)} \leq 1, \label{eq:DFSinf}\\
        \big\lvert D^{\bm{\mu}} \phi \big\rvert_{\mathrm{Lip}(\mathbb{R}^2,\R^3)} \leq 1, \label{eq:DFSL}\\
        \big\lvert D^{\bm{\mu}}\phi \big\rvert_{\mathcal{C}^{\alpha}(\mathbb{R}^2,\R^3)} \hspace{2pt} \leq 2 \label{eq:DFSH}.
    \end{align}
\end{lem}

\begin{proof}
    We note that $\phi_l$ is the product of sine or cosine for any $l \in [3]$.
    Therefore, all partial derivatives are also the product of sine, cosine or zero, which implies \eqref{eq:DFSinf}.
    We show the Lipschitz continuity of $\phi$.
    A mean value theorem for vector-valued functions in \cite[p.\ 113]{Rud76} states that for any $g\in\mathcal{C}^1(\R)$ there exists some $\tau\in[0,1]$ such that
    $$
    \norm{g(1)-g(0)}
    \le \norm{\nabla g(\tau)}.
    $$
    Let $\bm{x},\,\bm{h}\in\T^2$.
    For $g(t):= \phi(\bm{x}+t\bm{h})$, $t\in[0,1]$ and $\bm{v} := \bm{x}+\tau\bm{h}$, we see that
    $$
    \norm{\phi(\bm{x}+\bm{h})-\phi(\bm{x})}
    \le \norm{J \phi(\bm{v}) \, \bm{h}},
    $$
    where $J\phi$ denotes the Jacobian.
    We have
    \begin{equation*}
      \norm{J\phi (\bm{v})\, \bm{h}}^2
      = \norm{\begin{pmatrix}
        -\sin v_1\, \sin v_2 & \cos v_1\, \cos v_2\\
        \cos v_1\, \sin v_2 & \sin v_1\, \cos v_2\\
        0 & -\sin v_2
      \end{pmatrix}
      \bm{h}}^2
    = h_1^2\, \sin^2v_2 + h_2^2
    \le \norm{\bm{h}}^2.
    \end{equation*}
    This shows \eqref{eq:DFSL} for $\bm{\mu}=\bm{0}$.
    For general $\bm{\mu}\in\NN^2$,
    we note that $D^{\bm{\mu}}\phi$ takes only a limited number of different functions as derivatives of the product of sine or cosine.
    Since a changed sign does not affect the norm, we need to show \eqref{eq:DFSL} only for ${\bm{\mu}\in\{(1,0),(2,0),(0,1),(1,1),(2,1)\}}$,
    which is be done with the same arguments as for $\bm{\mu}=\bm{0}$ and therefore omitted here.
    Finally, the Hölder continuity \eqref{eq:DFSH} follows from the above and \Cref{prop:lipH}.
\end{proof}

\begin{thm}\label{thm:diffC}
    Let $k \in \NN$ and $f \in \mathcal{C}^{k+1}(\mathbb{S}^2)$. Then for all $0<\alpha<1$, the DFS function $\tilde{f}$ of $f$ is in $\mathcal{C}^{k,\alpha}(\mathbb{T}^2)$ and we have
    \begin{equation}\label{eq:diffC1}
        \big\lvert \tilde{f} \big\rvert_{\mathcal{C}^{k,\alpha}(\mathbb{T}^2)}\leq (k+3)!\, \lVert f \rVert_{\mathcal{C}^{k+1}(\mathbb{S}^2)},
    \end{equation}
    and
    \begin{equation}\label{eq:diffC2}
         \big\lVert \tilde{f} \big\rVert_{\mathcal{C}^{k,\alpha}(\mathbb{T}^2)}\leq \frac{7}{4} \, (k+3)!\, \lVert f \rVert_{\mathcal{C}^{k+1}(\mathbb{S}^2)}.       
    \end{equation}
\end{thm}

\begin{proof}
    Let $k \in \NN$ and $f \in \mathcal{C}^{k+1}(\mathbb{S}^2)$. 
    For $k^\prime \in \NN$ with $0<k^\prime \leq k+1$ and an open set $U\supset\S^2$, we consider a $\mathcal{C}^{k'}$-extension $f^*\in \mathcal{C}^{k^\prime}(U)$ of $f$. 
    We note that the DFS coordinate transform $\phi$ satisfies $\phi[\mathbb{R}^2] = \mathbb{S}^2 \subset U$, so we can consider it as function 
    $\phi\colon\R^2\to U$.

    Let $\bm{\beta}\in B^2_{k^\prime}$.
    By \Cref{lem:faaDi}, the DFS function $\tilde{f}$ is $k^\prime$-times continuously differentiable and there exist constants of \Cref{eq:indH1,eq:indH2,eq:indH3,eq:indH4,eq:indH5}, such that we have
    \begin{equation*}
        D^{\bm{\beta}} \tilde{f} = D^{\bm{\beta}} (f^* \circ \phi)= \sum_{i=1}^n (D^{\bm{\gamma}_i}f^*)\circ \phi \cdot \prod_{j=1}^{m_i} D^{\bm{\mu}_{i,j}}\phi_{l_{i,j}}.
    \end{equation*}
    Let  $\bm{x} \in \mathbb{R}^2$.
    It follows that
    \begin{align*}
        \big\lvert D^{\bm{\beta}}\tilde{f}(\bm{x})\big\rvert 
        & \leq  \sum_{i=1}^n \big\lvert (D^{\bm{\gamma}_i}f^*)\circ \phi(\bm{x})\big\rvert \, \prod_{j=1}^{m_i} \big\lvert D^{\bm{\mu}_{i,j}}\phi_{l_{i,j}}(\bm{x})\big\rvert .
    \intertext{By \eqref{eq:Ck-S}, \eqref{eq:indH1}, and \eqref{eq:DFSinf} , we obtain}
    		\big\lvert D^{\bm{\beta}}\tilde{f}(\bm{x})\big\rvert 
    		& \leq n \, \big\lvert f^* \big\rvert_{\mathcal{C}^{k^\prime}(\mathbb{S}^2)}^*
        \leq 2^{-1}\, (k^\prime+2)! \, \big\lvert f^* \big\rvert_{\mathcal{C}^{k^\prime}(\mathbb{S}^2)}^*.
    \end{align*}
    Note that this bound holds for any $\mathcal{C}^{k^\prime}$-extension $f^*$ of $f$. Hence, applying \Cref{defin:sphere}, we can replace $\big\lvert f^* \big\rvert_{\mathcal{C}^{k^\prime}(\mathbb{S}^2)}^*$ by $\lVert f \rVert_{\mathcal{C}^{k^\prime}(\mathbb{S}^2)}$ on the right hand side and the bound still holds. Therefore, we have for all $\bm{\beta} \in B^2_{k+1}$ that
    \begin{equation}\label{eq:hb1}
        \big\lVert D^{\bm{\beta}}\tilde{f}\big\rVert_{\mathcal{C}(\mathbb{R}^2)} \leq 2^{-1}\, (\lvert \bm{\beta}\rvert +2)! \, \lVert f \rVert_{\mathcal{C}^{\lvert \bm{\beta}\rvert}(\mathbb{S}^2)}.
    \end{equation}
    Furthermore, for $\bm{\beta}\in B^2_k$ we know that $D^{\bm{\beta}}\tilde{f}$ is continuously differentiable with
    \begin{equation*}
        \begin{alignedat}{2}
            & \big\lVert \nabla (D^{\bm{\beta}} \tilde{f})\big\rVert_{\mathcal{C}(\mathbb{R}^2)} \ && =\sup_{\bm{x}\in \mathbb{R}^2} \Big( \big\lvert D^{(\bm{e}^1+\bm{\beta})}\tilde{f}(\bm{x})\big\rvert^2+\big\lvert D^{(\bm{e}^2+\bm{\beta})}\tilde{f}(\bm{x})\big\rvert^2 \Big)^{\frac{1}{2}}\\
            & && \leq \Big(2\, \big(2^{-1}\, (k+3)! \, \lVert f \rVert_{\mathcal{C}^{k+1}(\mathbb{S}^2)}\big)^2\Big)^{\frac{1}{2}}\\
            & && = 2^{-\frac{1}{2}}\, (k+3)! \, \lVert f \rVert_{\mathcal{C}^{k+1}(\mathbb{S}^2)}.
        \end{alignedat}
    \end{equation*}
    By \eqref{eq:lipH2}, this implies that
    \begin{equation}\label{eq:hb2}
        \big\lvert D^{\bm{\beta}}\tilde{f}\big\rvert_{\mathrm{Lip}(\mathbb{R}^2)}
        \leq \sup_{\bm{x} \in \R^2} \big\lVert \nabla(D^{\bm{\beta}}\tilde{f})(\bm{x}) \big\rVert 
        \leq (k+3)! \, \lVert f \rVert_{\mathcal{C}^{k+1}(\mathbb{S}^2)}.
    \end{equation}
    Combining \eqref{eq:hb1} with \eqref{eq:hb2} and applying \eqref{eq:lipH1}, we conclude for all $\bm{\beta}\in B^2_{k}$ that $D^{\bm{\beta}}\tilde{f}$ is $\alpha$-Hölder continuous with
    \begin{equation*}
        \big\lvert D^{\bm{\beta}}\tilde{f}\big\rvert_{\mathcal{C}^\alpha(\mathbb{T}^2)}\leq \max\{\big\lvert D^{\bm{\beta}}\tilde{f}\big\rvert_{\mathrm{Lip}(\mathbb{T}^2)},2\big\lVert D^{\bm{\beta}}\tilde{f} \big\rVert_{\mathcal{C}(\mathbb{T}^2)}\}\leq (k+3)! \, \lVert f \rVert_{\mathcal{C}^{k+1}(\mathbb{S}^2)}.
    \end{equation*}
    Since the right hand side is independent of $\bm{\beta}$, it follows that $\tilde{f}$ is $(k,\alpha)$-Hölder continuous with
    \begin{equation*}
        \big\lvert \tilde{f}\big\rvert_{\mathcal{C}^{k,\alpha}(\mathbb{T}^2)}= \max_{\bm{\beta}\in B^2_k, \lvert \bm{\beta}\rvert =k} \big\lvert D^{\bm{\beta}}\tilde{f}\big\rvert_{\mathcal{C}^\alpha(\mathbb{T}^2)}\leq (k+3)! \, \lVert f \rVert_{\mathcal{C}^{k+1}(\mathbb{S}^2)},
    \end{equation*}
    which proves \eqref{eq:diffC1}.
    By \eqref{eq:hb1}, we have
    \begin{equation*}
        \big\lVert \tilde{f} \big\rVert_{\mathcal{C}^k(\mathbb{T}^2)} 
        =\sum_{\bm{\beta}\in B^2_k} \big\lVert D^{\bm{\beta}}\tilde{f}\big\rVert_{\mathcal{C}(\mathbb{R}^2)}\\
        \leq \sum_{i=0}^k \sum_{\bm{\beta}\in B^2_k,\, \lvert \bm{\beta}\rvert=i} \frac{1}{2} \, (i+2)! \, \lVert f \rVert_{\mathcal{C}^i(\mathbb{S}^2)}.
        \end{equation*}
      Splitting the first sum and noting that the inner sum comprises $i+1$ summands, we obtain
        \begin{align}
            \big\lVert \tilde{f} \big\rVert_{\mathcal{C}^k(\mathbb{T}^2)}
            &\leq \frac{1}{2} \lVert f \rVert_{\mathcal{C}^{k}(\mathbb{S}^2)}\, \Big((k+2)! (k+1)+(k+1)!\sum_{i=0}^{k-1}(i+1)\Big)\nonumber\\
            & = \frac{1}{2} \lVert f \rVert_{\mathcal{C}^k(\mathbb{S}^2)}\, (k+1)!(k+1) \left(k+1+\frac{k+2}{2}\right)\nonumber\\
            & \leq \frac{3}{4} (k+3)! \, \lVert f \rVert_{\mathcal{C}^k(\mathbb{S}^2)}.
            \label{eq:hb4}
    \end{align}
    This finishes the proof by
    \begin{equation*}
        \big\lVert \tilde{f} \big\rVert_{\mathcal{C}^{k,\alpha}(\mathbb{T}^2)} 
        =\big\lVert \tilde{f}\big\rVert_{\mathcal{C}^k(\mathbb{T}^2)}+\big\lvert \tilde{f} \big\rvert_{\mathcal{C}^{k,\alpha}(\mathbb{T}^2)}\\
        \leq \frac{7}{4}(k+3)! \, \lVert f \rVert_{\mathcal{C}^{k+1}(\mathbb{S}^2)}.
        \qedhere
    \end{equation*}
\end{proof}

\begin{cor}\label{cor:diffBd}
    Let $f \in \mathcal{C}^k(\mathbb{S}^2)$. Then the DFS function $\tilde{f}$ of $f$ is in $\mathcal{C}^k(\mathbb{T}^2)$ and by \eqref{eq:hb4} we have
    \begin{equation*} 
        \big\lVert \tilde{f} \big\rVert_{\mathcal{C}^k(\mathbb{T}^2)}\leq \frac{3}{4}(k+3)! \, \lVert f \rVert_{\mathcal{C}^k(\mathbb{S}^2)}.
    \end{equation*}
\end{cor}

Next, we establish that the DFS method even preserves the Hölder-class of a function, i.e., the weaker assumption of $(k,\alpha)$-Hölder continuity on the sphere already yields \text{${(k,\alpha)}$-Hölder} continuity of the DFS function on the torus.

\begin{thm}\label{thm:hoelC} 
    Let $k \in \NN$, $0<\alpha<1$ and $f \in \mathcal{C}^{k,\alpha}(\mathbb{S}^2)$. Then the DFS function $\tilde{f}$ of $f$ is in $\mathcal{C}^{k,\alpha}(\mathbb{T}^2)$. We have 
    \begin{equation}\label{eq:hoelC1}
        \big\lvert \tilde{f} \big\rvert_{\mathcal{C}^{k,\alpha}(\mathbb{T}^2)}\leq   (k+3)! \, \lVert f \rVert_{\mathcal{C}^{k,\alpha}(\mathbb{S}^2)}
    \end{equation}
    and
    \begin{equation}\label{eq:hoelC2}
        \big\lVert \tilde{f} \big\rVert_{\mathcal{C}^{k,\alpha}(\mathbb{T}^2)}\leq  \frac74 \, (k+3)! \, \lVert f \rVert_{\mathcal{C}^{k,\alpha}(\mathbb{S}^2)}.
    \end{equation}
\end{thm}

\begin{proof} 
    \underline{Case $k=0$}:
    Let $f \in \mathcal{C}^\alpha(\mathbb{S}^2)$ and $\bm{x}, \, \bm{y} \in \mathbb{R}^2$. Then
    \begin{equation*}
        \begin{alignedat}{2}
                & \big\lvert \tilde{f}(\bm{x})-\tilde{f}(\bm{y})\big\rvert \ && = \lvert f(\phi(\bm{x}))-f(\phi(\bm{y}))\rvert\\
                & && \leq \lvert f \rvert_{\mathcal{C}^\alpha(\mathbb{S}^2)} \,  \lVert \phi(\bm{x})-\phi(\bm{y})\rVert^\alpha\\
                & &&\leq \lvert f \rvert_{\mathcal{C}^\alpha(\mathbb{S}^2)} \, \lvert \phi \rvert_{\mathrm{Lip}(\mathbb{R}^2,\R^3)}^\alpha \, \lVert \bm{x}-\bm{y} \rVert^\alpha.
        \end{alignedat}
    \end{equation*}
    We apply \eqref{eq:DFSL} to deduce that $\tilde{f}\in \mathcal{C}^\alpha(\mathbb{R}^2)$ with 
    \begin{equation*}
        \big\lvert \tilde{f} \big\rvert_{\mathcal{C}^\alpha(\mathbb{T}^2)}
        \leq \lvert f \rvert_{\mathcal{C}^\alpha(\mathbb{S}^2)},
    \end{equation*}
    which shows \eqref{eq:hoelC1}.
    Furthermore, \eqref{eq:hoelC2} holds since
    \begin{equation*}
        \big\lVert \tilde{f} \big\rVert_{\mathcal{C}^{\alpha}(\mathbb{T}^2)}=  \big\lVert \tilde{f}\big\rVert_{\mathcal{C}(\mathbb{R}^2)}+\big\lvert \tilde{f}\big\rvert_{\mathcal{C}^{\alpha}(\mathbb{R}^2)}\leq \lVert f \rVert_{\mathcal{C}(\mathbb{S}^2)}+ \lvert f \rvert_{\mathcal{C}^{\alpha}(\mathbb{S}^2)}
        \leq \frac74 \, 3! \, \lVert f \rVert_{\mathcal{C}^{\alpha}(\mathbb{S}^2)}.
    \end{equation*}
  
    \underline{Case $k>0$}:
    Let $k \in \mathbb{N}$ and $f \in \mathcal{C}^{k,\alpha}(\mathbb{S}^2)$. Let $f^\ast\in \mathcal{C}^{k,\alpha}(U)$ be a $(k,\alpha)$-Hölder extension of $f$. 
    Furthermore, let $\bm{\beta} \in B^2_k$. We apply \Cref{lem:faaDi} to obtain
    \begin{equation*}
        D^{\bm{\beta}} \tilde{f} = D^{\bm{\beta}} (f^* \circ \phi)= \sum_{i=1}^n (D^{\bm{\gamma}_i}f^*)\circ \phi \cdot \prod_{j=1}^{m_i} D^{\bm{\mu}_{i,j}}\phi_{l_{i,j}}
    \end{equation*}
    with constants from \Cref{eq:indH1,eq:indH2,eq:indH3,eq:indH4,eq:indH5}.
    We order the terms such that, for some $n^\prime \leq n$, we have $\lvert \bm{\gamma}_i\rvert < k$ for $i\in [n^\prime]$ and $\lvert\bm{\gamma}_i\rvert=k$ for $i \in [n]\setminus [n^\prime]$. Since $f^*\in\mathcal{C}^{k,\alpha}(U)$, we know that $D^{\bm{\gamma}_i}f^*$ is continuously differentiable for all $i \in [n^\prime]$. Clearly, $D^{\bm{\gamma}_i}f^*$ is a $\mathcal{C}^1$-extension of the restriction $\restr{(D^{\bm{\gamma}_i}f^*)}{\mathbb{S}^2}$ and hence we have $ D^{\bm{\gamma}_i}f^* \in \mathcal{C}^1(\mathbb{S}^2)$ with 
    \begin{equation*}
        \begin{alignedat}{2}
             & \big\lVert D^{\bm{\gamma}_i}f^*\big\rVert_{\mathcal{C}^1(\mathbb{S}^2)}
             &&\leq \big\lvert D^{\bm{\gamma}_i}f^* \big\rvert_{\mathcal{C}^1(\mathbb{S}^2)}^*=\sum_{\bm{\beta}\in B^3_1}\big\lVert D^{\bm{\beta}}\big(D^{\bm{\gamma}_i}f^*\big)\big\rVert_{\mathcal{C}(\mathbb{S}^2)}\\
             & &&\leq \sum_{\bm{\beta}\in B^3_k} \big\lVert D^{\bm{\beta}} f^*\big\rVert_{\mathcal{C}(\mathbb{S}^2)}=\big\lvert f^* \big\rvert_{\mathcal{C}^{k}(\mathbb{S}^2)}^*\leq \big\lvert f^* \big\rvert_{\mathcal{C}^{k,\alpha}(\mathbb{S}^2)}^*,
        \end{alignedat}
    \end{equation*} 
    for all $i \in [n^\prime]$. We obtain by \eqref{eq:diffC1} that
    \begin{equation}\label{eq:hoelextbd}
        \big\lvert D^{\bm{\gamma}_i}f^* \circ \phi\big\rvert_{\mathcal{C}^\alpha(\mathbb{T}^2)}\leq (0+3)! \, \lVert D^{\bm{\gamma}_i}f \rVert_{\mathcal{C}^1(\mathbb{S}^2)}\leq 6 \, \lvert f^* \rvert_{\mathcal{C}^{k,\alpha}(\mathbb{S}^2)}^*.
    \end{equation}
    
    Let $\bm{x},\bm{y} \in \mathbb{R}^2$. Then, by \eqref{eq:indH6}, we have
    \begin{equation*}
        \begin{alignedat}{2}
            & && \big\lvert D^{\bm{\beta}} \tilde{f} (\bm{x})-D^{\bm{\beta}} \tilde{f}(\bm{y}) \big\rvert \\
            &= \ && \Big\lvert \sum_{i=1}^n (D^{\bm{\gamma}_i}f^*)\circ \phi(\bm{x}) \cdot \prod_{j=1}^{m_i} D^{\bm{\mu}_{i,j}}\phi_{l_{i,j}}(\bm{x}) - \sum_{i=1}^n (D^{\bm{\gamma}_i}f^*)\circ \phi(\bm{y}) \cdot \prod_{j=1}^{m_i} D^{\bm{\mu}_{i,j}}\phi_{l_{i,j}}(\bm{y}) \Big \rvert
        \end{alignedat}
    \end{equation*}
    hence,
    \begin{equation*}
      \begin{alignedat}{2}
        & \big\lvert D^{\bm{\beta}} \tilde{f} (\bm{x})-D^{\bm{\beta}} \tilde{f}(\bm{y}) \big\rvert 
        \leq &&  \underbrace{\sum_{i=1}^n \big\lvert (D^{\bm{\gamma}_i}f^*)\circ \phi(\bm{x}) - (D^{\bm{\gamma}_i}f^*)\circ \phi(\bm{y}) \big\rvert \, \Big\lvert \prod_{j=1}^{m_i} D^{\bm{\mu}_{i,j}}\phi_{l_{i,j}}(\bm{x}) \Big\rvert}_{=:A} \\
        & && +  \underbrace{\sum_{i=1}^n  \big\lvert(D^{\bm{\gamma}_i}f^*)\circ \phi(\bm{y})\big\rvert \, \Big\lvert \prod_{j=1}^{m_i} D^{\bm{\mu}_{i,j}}\phi_{l_{i,j}}(\bm{x})- \prod_{j=1}^{m_i} D^{\bm{\mu}_{i,j}}\phi_{l_{i,j}}(\bm{y})\Big \rvert.}_{=:B}
      \end{alignedat}
    \end{equation*}
    We estimate the first sum as
    \begin{equation*}
        \begin{alignedat}{2}
            A & \underset{\mathrm{\eqref{eq:DFSinf}}}{\leq}  && \sum_{i=1}^n \big \lvert (D^{\bm{\gamma}_i}f^*)\circ \phi(\bm{x}) - (D^{\bm{\gamma}_i}f^*)\circ \phi(\bm{y}) \big \rvert \\
            & \underset{{\eqref{eq:hoelextbd}}}{\leq} && \sum_{i=1}^{n^\prime} 6 \, \lvert f^* \rvert_{\mathcal{C}^{k,\alpha}(\mathbb{S}^2)}^* \, \lVert \bm{x}-\bm{y}\rVert^\alpha + \sum_{i=n^\prime +1}^n \lvert f^*\rvert_{\mathcal{C}^{k,\alpha}(\mathbb{S}^2)}^* \, \lVert \phi(\bm{x}) - \phi(\bm{y}) \rVert^\alpha\\
            & \underset{\mathrm{\eqref{eq:DFSL}}}{\leq}&& \lvert f^* \rvert_{\mathcal{C}^{k,\alpha}(\mathbb{S}^2)}^* \, \Big( \sum_{i=1}^{n^\prime} 6 \, \lVert \bm{x}-\bm{y}\rVert^\alpha + \sum_{i=n^\prime +1}^n \lVert \bm{x} -\bm{y} \rVert^\alpha \Big)\\
            & \underset{\mathrm{\eqref{eq:indH1}}}{\leq} && {3 \,(k+2)! \, \lvert f^* \rvert_{\mathcal{C}^{k,\alpha}(\mathbb{S}^2)}^*} \, \lVert \bm{x} - \bm{y} \rVert^\alpha.
        \end{alignedat}
    \end{equation*}
    By \eqref{eq:Ck-S}, we estimate the second sum as
    \begin{equation*}
        \begin{aligned}
            B={}& \sum_{i=1}^n  \big\lvert(D^{\bm{\gamma}_i}f^*)\circ \phi(\bm{y})\big \rvert \, \Big\lvert \prod_{j=1}^{m_i} D^{\bm{\mu}_{i,j}}\phi_{l_{i,j}}(\bm{x})- \prod_{j=1}^{m_i} D^{\bm{\mu}_{i,j}}\phi_{l_{i,j}}(\bm{y})\Big \rvert\\
            {\leq} & \sum_{i=1}^n \lvert f^* \rvert_{\mathcal{C}^{k,\alpha}(\mathbb{S}^2)}^*  \, \Big\lvert \prod_{j=1}^{m_i} D^{\bm{\mu}_{i,j}}\phi_{l_{i,j}}(\bm{x})- \prod_{j=1}^{m_i} D^{\bm{\mu}_{i,j}}\phi_{l_{i,j}}(\bm{y})\Big \rvert.
        \end{aligned}
    \end{equation*}
  	Using a telescoping sum, we write last equation as
    \begin{multline*}
        B\leq  \sum_{i=1}^n \lvert f^* \rvert_{\mathcal{C}^{k,\alpha}(\mathbb{S}^2)}^* 
        \\\cdot \Big \lvert \sum_{h=1}^{m_i} \big( D^{\bm{\mu}_{i,h}}\phi_{l_{i,h}}(\bm{x})-D^{\bm{\mu}_{i,h}}\phi_{l_{i,h}}(\bm{y}) \big) \prod_{j=1}^{h-1} D^{\bm{\mu}_{i,j}}\phi_{l_{i,j}}(\bm{x}) \prod_{j=h+1}^{m_i} D^{\bm{\mu}_{i,j}}\phi_{l_{i,j}}(\bm{y}) \Big\rvert.
    \end{multline*}
    By \eqref{eq:DFSinf}, we further estimate 
    \begin{align*}
        B \leq{}& \sum_{i=1}^n \sum_{h=1}^{m_i} \lvert f^* \rvert_{\mathcal{C}^{k,\alpha}(\mathbb{S}^2)}^* \, \big \lvert D^{\bm{\mu}_{i,h}}\phi_{l_{i,h}}(\bm{x})-D^{\bm{\mu}_{i,h}}\phi_{l_{i,h}}(\bm{y}) \big\rvert\\
        \leq{} & \sum_{i=1}^n \sum_{h=1}^{m_i} \lvert f^* \rvert_{\mathcal{C}^{k,\alpha}(\mathbb{S}^2)}^* \, \lVert D^{\bm{\mu}_{i,h}}\phi(\bm{x})-D^{\bm{\mu}_{i,h}}\phi(\bm{y}) \rVert.
    \intertext{By \eqref{eq:DFSH} as well as \eqref{eq:indH1} and \eqref{eq:indH3}, we have }
        B \leq{} & \sum_{i=1}^n \sum_{h=1}^{m_i} \lvert f^* \rvert_{\mathcal{C}^{k,\alpha}(\mathbb{S}^2)}^* \, 2 \, \lVert \bm{x}-\bm{y} \rVert^\alpha\\
        \leq{} & {(k+2)! \, k \, \lvert f^* \rvert_{\mathcal{C}^{k,\alpha}(\mathbb{S}^2)}^*} \, \lVert \bm{x}-\bm{y} \rVert^\alpha.
    \end{align*}
    Combining these upper bounds on $A$ and $B$, we obtain
    \begin{equation*}
        \lVert D^{\bm{\beta}} \tilde{f}(\bm{x})-D^{\bm{\beta}} \tilde{f}(\bm{y}) \rVert 
        \leq { (k+3)! \, \lvert f^* \rvert_{\mathcal{C}^{k,\alpha}(\mathbb{S}^2)}^* } \, \lVert \bm{x}-\bm{y} \rVert^\alpha.
    \end{equation*}
    Since the last equation holds independently of the choices of $\bm{x}$, $\bm{y}$, and $\bm{\beta}$, we see that 
    \begin{equation*}
        \big\lvert \tilde{f} \big\rvert_{\mathcal{C}^{k,\alpha}(\mathbb{T}^2)}\leq (k+3)! \, \lvert f^*\rvert_{C^{k,\alpha}(\mathbb{S}^2)}^*.
    \end{equation*}
    By \Cref{defin:sphere}, this bound still holds if we replace $\lvert f^*\rvert_{C^{k,\alpha}(\mathbb{S}^2)}^*$ by $\lVert f\rVert_{C^{k,\alpha}(\mathbb{S}^2)}$ on the right hand side, since the $\mathcal{C}^{k,\alpha}$-extension $f^*$ was chosen arbitrarily, which yields \eqref{eq:hoelC1}.
    Since $f^*$ is a $\mathcal{C}^{k,\alpha}$-extension of $f$ and thus also a $\mathcal{C}^k$-extension of $f$, \Cref{cor:diffBd} in combination with \Cref{defin:sphere} show that
    \begin{equation*}
        \big\lVert \tilde{f} \big\rVert_{\mathcal{C}^k(\mathbb{T}^2)}\leq \frac{3}{4}(k+3)! \, \lVert f \rVert_{\mathcal{C}^k(\mathbb{S}^2)} \leq \frac{3}{4}(k+3)! \, \lvert f^* \rvert_{\mathcal{C}^{k,\alpha}(\mathbb{S}^2)}^*.
    \end{equation*}
    Hence, \eqref{eq:hoelC2} follows with
    \begin{equation*}
        \big\lVert \tilde{f} \big\rVert_{\mathcal{C}^{k,\alpha}(\mathbb{T}^2)}=\big\lVert \tilde{f}\big\rVert_{\mathcal{C}^k(\mathbb{T}^2)}+\big\lvert \tilde{f}\big\rvert_{\mathcal{C}^{k,\alpha}(\mathbb{T}^2)}\leq \left({\frac{3}{4}+1}\right)\, (k+3)! \, \lVert f \rVert_{\mathcal{C}^{k,\alpha}(\mathbb{S}^2)}.\qedhere
    \end{equation*}
\end{proof}
\section{Fourier series of Hölder continuous DFS functions}\label{chap:3}

In this section, we combine our findings from \Cref{thm:diffC,thm:hoelC}
with results from multi-dimensional Fourier analysis to obtain convergence results on the Fourier series of DFS functions.

\subsection{Fourier series of DFS functions}

We want to approximate spherical functions via the Fourier series of their DFS functions.
We first recall Fourier series on the torus $\T^2$.
We denote by $L_2(\mathbb{T}^2)$ the space of square-integrable functions $f\colon\mathbb{T}^2 \to \mathbb{C}$ with the norm
$
  \lVert f\rVert_{L_2(\mathbb{T}^2)}^2
  :=\int_{\T^2} \lvert f(\bm{x})\rvert^2 \,\mathrm{d}\bm{x}.
$

For $c_{\bm{n}}\in\C$, $\bm{n}\in \mathbb{Z}^2$, we call the series $\sum_{\bm{n}\in \mathbb{Z}^2} c_{\bm{n}}$ {convergent} whenever for all expanding sequences $\{\Omega_h\}_{h\in \mathbb{N}}$ of bounded sets exhausting $\mathbb{Z}^2$ the partial sums $\sum_{\bm{n} \in \Omega_h} c_{\bm{n}}$ converge absolutely as $h \to \infty$, cf.\ \cite[p.~6]{comHarAna}.

\begin{defin}\label{defin:F_torus}
    Let $g \in L_2 (\mathbb{T}^2)$ and $\bm{n} \in \mathbb{Z}^2$. We define the $\bm{n}$-th \emph{Fourier coefficient} of $g$ by
    \begin{equation*}
        c_{\bm{n}}(g):=(2\pi)^{-2} \int_{\mathbb{T}^2} g(\bm{x})\, \e^{-\i \langle \bm{n},\bm{x}\rangle}\, \mathrm{d}\bm{x}.
    \end{equation*}
    Let $\Omega_h$, $h\in \mathbb{N}$, be an expanding sequence of bounded sets which exhausts $\mathbb{Z}^2$. 
    We define the $h$-th \emph{partial Fourier sum} of $g$ by
    \begin{equation} \label{eq:F_torus}
        F_{\Omega_h} g(\bm{x}):=\sum_{\bm{n} \in \Omega_h} c_{\bm{n}}(g)\, \e^{\i \langle \bm{n},\bm{x} \rangle},\qquad \bm{x} \in \mathbb{T}^2,
    \end{equation}
    and the \emph{Fourier series}
    $ F g := \lim_{h\to\infty} F_{\Omega_h} g $.
\end{defin}

The functions $\e^{\i\inn{\bm{n},\cdot}}$, $\bm{n}\in\Z^2$, form an orthogonal basis of the Hilbert space $L_2(\T^2)$.
Hence, we have $Ff=f$ for all $f\in L_2(\T^2)$.
The Fourier sum \eqref{eq:F_torus} can be evaluated efficiently with the fast Fourier transform (FFT). 

\begin{rem}
On the sphere $\S^2$,
the space $L_2(\S^2)$ consists of all square-integrable functions $f \colon\mathbb{S}^2 \to \mathbb{C}$ with respect to
\begin{equation}\label{eq:intS}
  \int_{\mathbb{S}^2} f(\bm{\xi})\,\mathrm{d} \bm{\xi}
  = \int_{-\pi}^{\pi}  \int_{0}^\pi \tilde{f}(\lambda,\theta) \, \sin(\theta) \,\mathrm{d}\theta \,\mathrm{d}\lambda.
\end{equation}
An orthogonal basis of $L_2(\S^2)$ is given by the \emph{spherical harmonics} 
$$
Y_n^k(\phi(\lambda,\theta))
= P_n^k(\cos\theta)\, \e^{\i k\lambda},
\qquad \lambda \in [-\pi,\pi),\ \theta\in[0,\pi],
$$
where $P_n^k$ is the associated Legendre function of degree $n$ and order $k$,
see \cite[sec.~5]{michel}.
Any spherical function $f\in L_2(\S^2)$ can be written as \emph{spherical Fourier series}
\begin{equation}\label{eq:sh-series}
  f(\zb\xi) = \sum_{n=0}^{\infty} \sum_{k=-n}^{n} \hat f_{n,k}\,  Y_n^k(\zb\xi),
  \qquad \zb\xi\in\S^2,
\end{equation}
with some coefficients $\hat f_{n,k}\in\C$.
Contrary to \eqref{eq:F_torus},
the sums over $n$ and $k$ in \eqref{eq:sh-series} cannot be separated
because the associated Legendre functions $P_n^k$ depend on both summation indices,
which makes the computation more time and memory consuming.
There are fast spherical Fourier algorithms for the computation of \eqref{eq:sh-series}, see, e.g., \cite{drhe,kunis,Mo99}.
However, 
they are not as fast as FFTs of a comparable size 
and they suffer from the problem that
associated Legendre functions $P_n^n$ of order $n\gtrapprox 1500$ can be too small to be representable in double precision, cf.\ \cite{Sch13}.\qed
\end{rem}

The DFS method
represents a spherical function~$f$ 
via the Fourier series of its DFS function $\tilde f\colon\S^2\to\C$, i.e.,
\begin{equation}\label{eq:DFSseries1}
  F_{\Omega_h}\tilde f(\bm{x}) = 
  \sum_{\bm{n}\in\Omega_h} c_{\bm{n}}(\tilde f)\, \e^{\i \inn{\bm{n},\bm x}}
  ,\qquad \bm x\in\T^2.
\end{equation} 
Let $\bm{n}=(n_1,n_2)\in\Z^2$ and $\bm{x}=(x_1,x_2)\in\T^2$.
We denote by 
$$
M(n_1,n_2) := (n_1,-n_2)
$$
the reflection in the second component.
We have
$$
\e^{\i \langle\bm{n}, (x_1+\pi,-x_2)\rangle}
=(-1)^{n_1}\, \e^{\i \langle M(\bm{n}), (x_1,x_2)\rangle}.
$$
Hence, the basis functions $\e^{\i \langle\bm{n},\cdot\rangle}$ on the torus $\T^2$ are not BMC functions \eqref{eq:BMC2} if $n_2\neq 0$
and thus cannot be directly transferred to the sphere.
However, it follows that that
\begin{equation} \label{eq:e}
  e_{\bm{n}}(\bm{x})
  :=  
  \begin{cases}
    \e^{\i \langle\bm{n}, \bm{x}\rangle} + (-1)^{n_1} \e^{\i \langle M(\bm{n}), \bm{x}\rangle} ,
    & n_2\neq 0,\\
    \e^{\i \langle\bm{n}, \bm{x}\rangle}, & n_2=0,
  \end{cases}
\end{equation}
is a BMC function.
We denote by $\phi^{-1}$ the well-defined inverse of the longitude-latitude transform $\phi$ on $[-\pi,\pi)\times(0,\pi)\cup\{(0,0),(0,\pi)\}$.
For $\bm{n}\in\Z\times\NN$, we define the basis functions 
$$
b_{\bm{n}}(\bm{\xi}) 
\coloneqq e_{\bm{n}}(\phi^{-1}(\bm{\xi}))
,\qquad 
\bm{\xi}\in \S^2.
$$
Since $e_{\bm{n}}$ is a BMC function, we have 
$
\tilde b_{\bm{n}}(\bm{x}) = e_{\bm{n}}(\bm{x})
$
for all 
$\bm{x}\in\T^2$ with $x_2\neq m\pi$, $m\in\Z$.
This motivates the following definition of an analogue of the Fourier series for the DFS method.

\begin{defin}\label{defin:F_sphere}
  Let $f\colon\mathbb{S}^2\to\mathbb{C}$ with the associated DFS function $\tilde{f}\in L_2(\mathbb{T}^2)$,
  and let $\{\Omega_h\}_{h\in \mathbb{N}}$ be an expanding sequence of bounded sets which exhausts $\mathbb{Z}\times\NN$. For $h \in \mathbb{N}$, we define the $h$-th \emph{partial DFS Fourier sum} of $f$ by
  \begin{equation*}
    S_{\Omega_h} f(\bm{\xi}):=
    \sum_{\bm{n}\in \Omega_h} c_{\bm{n}}(\tilde{f})\, b_{\bm{n}}(\zb\xi)
    ,\qquad \zb\xi\in\S^2,
  \end{equation*}
  and the \emph{DFS Fourier series} of $f$ by
  \begin{equation*}
    S f(\bm{\xi}):= \lim_{h \to \infty} S_{\Omega_h}f(\bm{\xi})
    ,\qquad \bm{\xi}\in\S^2.
  \end{equation*}
\end{defin}

The connection with the classical Fourier series is shown in the following theorem. 
We denote by $\tilde L_2(\S^2)$ the weighted space of measurable functions $f\colon\S^2\to\C$ with finite norm
$$
\norm{f}_{\tilde L_2(\S^2)}^2
:= \int_{\S^2} \abs{f(\bm{\xi})}^2 \frac{1}{\sqrt{1-\xi_3^2}} \d\bm{\xi}.
$$

\begin{thm} \label{thm:F}
  Let $f\in\tilde L_2(\S^2)$.
  Then the Fourier coefficients of its DFS function $\tilde f\in L_2(\T^2)$
  satisfy 
  $$
  c_{\bm{n}}(\tilde f) = (-1)^{n_1} c_{M(\bm{n})}(\tilde f)
  ,\qquad \bm{n}\in\Z^2.
  $$
  For $\Omega\in\Z\times\NN$, we set $\tilde\Omega := \Omega \cup M(\Omega) \subset\Z^2$.
  Then we have
  $$
  S_\Omega f (\bm\xi)
  =F_{\tilde\Omega} \tilde f (\phi^{-1}(\bm\xi)) 
  =  \sum_{\bm{n}\in\Omega} c_{\bm{n}}(\tilde f)\, e_{\bm{n}}(\phi^{-1}(\bm\xi))
  ,\qquad \bm\xi\in\S^2.
  $$
  The set $\{ b_{\bm{n}} \mid \bm{n}\in\Z\times\NN \}$ is an orthogonal basis of $\tilde L_2(\S^2)$.
\end{thm}
\begin{proof}
  The fact that $\tilde{f}\in L_2(\T^2)$ follows by the definition of $\phi$ and the spherical measure \eqref{eq:intS}.
  Let $\bm{n}\in\Z$.
  By \eqref{eq:BMC}, we have
  \begin{align*}
    c_{\bm{n}}(\tilde f)
    &= \int_{\T^2} \tilde f(x_1,x_2)\, \e^{\i (n_1x_1+n_2x_2)} \d\bm{x}
    \\&
    = (-1)^{n_1} \int_{\T^2} \tilde f(x_1+\pi,-x_2)\, \e^{\i (n_1(x_1+\pi)+n_2x_2)} \d\bm{x}.
  \end{align*}
  With the substitution $\T^2\ni (x_1,x_2)\mapsto (x_1+\pi,-x_2)\in\T^2$, we obtain
  \begin{equation*}
    c_{\bm{n}}(\tilde f)
    = (-1)^{n_1} \int_{\T^2} \tilde f(x_1,x_2)\, \e^{\i (n_1x_1-n_2x_2)} \d\bm{x}
    = (-1)^{n_1} c_{M(\bm{n})}(\tilde f).
  \end{equation*}
  Let $\bm x\in\T^2$.
  Then, we can split the Fourier sum of $\tilde f$ and obtain
  by \eqref{eq:e} that
  \begin{align*}
    F_{\tilde\Omega} \tilde f(\bm{x})
    ={}& \sum_{\bm{n}\in \tilde\Omega} c_{\bm{n}} (\tilde f)\, \e^{\i \langle \bm{n}, \bm{x}\rangle}
    \\
    ={}& \sum_{\substack{\bm{n}\in \Omega,\, n_2\neq0}} c_{\bm{n}} (\tilde f) \left(\e^{\i \langle\bm{n}, \bm{x}\rangle} 
    + (-1)^{n_1} \e^{\i \langle M(\bm{n}), \bm{x}\rangle} \right)
    + \sum_{\substack{\bm{n}\in \Omega,\, n_2=0}} c_{\bm{n}} (\tilde f)\, \e^{\i \langle \bm{n}, \bm{x}\rangle}
    \\
    ={}&
    \sum_{\bm{n}\in\Omega} c_{\bm{n}}(\tilde f)\, e_{\bm{n}} (\bm{x})
    .\qedhere
  \end{align*}
\end{proof}

\subsection{Convergence of the Fourier series}

We prove convergence results on the DFS Fourier series of \Cref{defin:F_sphere}. 
To simplify the notation, we set $\mathcal{C}^{k,1}(\S^2) := \mathcal{C}^{k+1}(\S^2)$ for $k\in\NN$, and the same for $\S^2$ replaced by~$\T^2$.

\begin{lem}\label{lem:bSeries}
    Let $k \in \NN, \ 0<\alpha \leq 1$, and $f\in \mathcal{C}^{k,\alpha}(\S^2)$. 
    Then, for all $b \in \mathbb{R}$ with $b > 2 / ({1+k+\alpha})$, the series
    \begin{equation}\label{eq:bSeries}
        \sum_{\bm{n} \in \mathbb{Z}^2} \big\lvert c_{\bm{n}}(\tilde{f})\big\rvert^b
    \end{equation}
    converges, where $c_{\bm{n}}(\tilde{f})$ denotes the Fourier coefficients of the DFS function $\tilde f$ of $f$.
\end{lem}

\begin{proof}
  Let $\alpha<1$.
  By \Cref{thm:hoelC}, we have $\tilde{f} \in \mathcal{C}^{k,\alpha}(\mathbb{T}^2)$.
  By \cite[p.~87]{comHarAna}, the fact that $\tilde{f} \in \mathcal{C}^{k,\alpha}(\mathbb{T}^2)$ directly implies the convergence of \eqref{eq:bSeries} for all $b> 2/({1+k+\alpha})$. 
  For $\alpha=1$, we choose $\varepsilon>0$ such that $b> {2}/({1+k+(1-\varepsilon)})$. 
  Then, \Cref{thm:diffC} shows that $\tilde{f}\in \mathcal{C}^{k,1-\varepsilon}(\mathbb{T}^2)$, which implies the convergence of \eqref{eq:bSeries} as above.
\end{proof}

\begin{thm}\label{thm:hoelScon}
    Let $k \in \mathbb{N}$, $0 < \alpha \leq 1$, and $f\in \mathcal{C}^{k,\alpha}(\mathbb{S}^2)$. 
    Then the Fourier series $F \tilde{f}$ converges to the DFS function $\tilde{f}$ uniformly on $\mathbb{T}^2$
    and, for $\Omega\subset\Z^2$, we have
    \begin{equation*}
        \big\lVert \tilde{f} - F_\Omega \tilde{f} \big\rVert_{\mathcal{C}(\mathbb{T}^2)} 
        \leq \sum_{\bm{n} \in \mathbb{Z}^2\setminus \Omega} \big\lvert c_{\bm{n}}(\tilde{f}) \big\rvert.
    \end{equation*}
  Further, the DFS Fourier series $Sf$ converges to $f$ uniformly on $\mathbb{S}^2$
  and for ${\Omega\in\Z\times\NN}$, we have
  \begin{equation*}
  \lVert f - S_{\Omega}f  \rVert_{\mathcal{C}(\mathbb{S}^2)} \leq \sum_{\bm{n} \in \mathbb{Z}^2\setminus \tilde \Omega} \big\lvert c_{\bm{n}}(\tilde{f}) \big\rvert.
  \end{equation*} 
\end{thm}

\begin{proof}
    By \Cref{lem:bSeries} with $b=1>2/(2 +\alpha)\geq 2/(1+k+\alpha)$, 
    the series
    $
      \sum_{\bm{n} \in \mathbb{Z}^2} \big\lvert c_{\bm{n}}(\tilde{f}) \big\rvert 
    $
    converges.
    By \cite[Theorem 1.37]{numFou}, it follows that the Fourier series $F\tilde f$ converges uniformly to
    \begin{equation*}
        \tilde{f}(\bm{x})=\sum_{\bm{n}\in \mathbb{Z}^2} c_{\bm{n}}(\tilde{f})\, \e^{\i \langle \bm{n},\bm{x} \rangle}
        ,\qquad \bm{x}\in\T^2.
    \end{equation*}
    Then, for all $\bm{x} \in \mathbb{T}^2$ and $\Omega\subset\Z^2$, we have
    \begin{equation*}
      \big\lvert \tilde{f} (\bm{x})- F_\Omega \tilde{f}(\bm{x}) \big\rvert
      = \Big\lvert \sum_{\bm{n} \in \mathbb{Z}^2} c_{\bm{n}}(\tilde{f})\, \e^{\i \langle \bm{n},\bm{x}\rangle} -\sum_{\bm{n} \in \Omega}c_{\bm{n}}(\tilde{f})\, \e^{\i \langle \bm{n},\bm{x}\rangle}\Big\rvert
      \leq \sum_{\bm{n} \in \mathbb{Z}^2\setminus \Omega} \big\lvert c_{\bm{n}}(\tilde{f})\, \e^{\i \langle \bm{n},\bm{x}\rangle}\big\rvert.
    \end{equation*}
  The second part follows with \Cref{thm:F}.
\end{proof}

In the following, we prove bounds on the Fourier coefficients of $\tilde{f}$. 

\begin{lem}\label{lem:bdC}
    Let $k \in \NN$, $0<\alpha<1$, and $g \in \mathcal{C}^{k,\alpha}(\mathbb{T}^2)$. Then it holds for all $\bm{n} \in \mathbb{Z}^2\setminus \{0\}$ that
    \begin{equation*}
        \lvert c_{\bm{n}}(g) \rvert \leq 2^{\frac{k+\alpha}{2}-1} \pi^\alpha \, \lvert g \rvert_{\mathcal{C}^{k,\alpha}(\mathbb{T}^2)} \, \frac{1}{\lvert \bm{n} \rvert^{k+\alpha}}.
    \end{equation*}
\end{lem}
\begin{proof}
  The lemma was proven in \cite[p.~180]{claFouAna} for $1$-periodic functions.
  We transfer this result by setting the $1$-periodic function $g_1(\bm{x}):=g(2\pi \bm{x})$.
  A change of variables shows that the Fourier coefficients of $g$ and $g_1$ coincide,
  $$
  c_{\bm{n}}(g) 
  = (2\pi)^{-1} \int_{[0,2\pi]^2} g(\bm{x})\, \e^{-\i\langle\bm{n},\bm{x}\rangle} \d\bm{x}
  = \int_{[0,1]^2} g_1(\bm{x})\, \e^{-2\pi\i\langle\bm{n},\bm{x}\rangle} \d\bm{x}.
  $$
  Furthermore, we see that 
  $ \abs{g_1}_{\mathcal{C}^{k,\alpha}} = (2\pi)^{k+\alpha} \abs{g}_{\mathcal{C}^{k,\alpha}}.$
\end{proof}

\begin{thm}\label{thm:bdC2}
    Let $k \in \NN$, $0<\alpha \leq1$, and $f\in \mathcal{C}^{k,\alpha}(\mathbb{S}^2)$. 
    Then, for all $\bm{n} \in \mathbb{Z}^2\setminus \{\bm0\}$, the Fourier coefficients of the DFS function $\tilde{f}$ are bounded by
    \begin{equation*}
        \big\lvert c_{\bm{n}}(\tilde{f}) \big\rvert \leq 2^{\frac{k+\alpha}{2}-1} \pi^\alpha \, \frac{(k+3)!}{\lvert \bm{n} \rvert^{k+\alpha}}\, \norm{f}_{\mathcal{C}^{k,\alpha}(\mathbb{S}^2)}.
    \end{equation*}
\end{thm}

\begin{proof}
  We first show the assertion for $\alpha<1$.
  By \Cref{thm:hoelC}, we know that
  $\tilde{f} \in \mathcal{C}^{k,\alpha}(\mathbb{T}^2)$ with $\lvert \tilde{f}\rvert_{\mathcal{C}^{k,\alpha}(\mathbb{T}^2)}\leq (k+3)!\, \lVert f \rVert_{\mathcal{C}^{k,\alpha}(\S^2)}$.
  Then, \Cref{lem:bdC}
  implies the statement for $\alpha<1$.
  For $\alpha=1$, \Cref{thm:diffC} yields that
  $\tilde{f} \in \mathcal{C}^{k,1-\epsilon}(\mathbb{T}^2)$ with $\big\lvert \tilde{f}\big\rvert_{\mathcal{C}^{k,1-\epsilon}(\mathbb{T}^2)}\leq (k+3)!\, \norm{f}_{\mathcal{C}^{k,\alpha}(\mathbb{S}^2)}$ for all $0<\epsilon<1$.
  The claim follows from the first part combined with the fact that $2^{\frac{k+\alpha}{2}-1} \pi^\alpha$ is continuous in $\alpha$.
\end{proof}

\begin{thm}\label{thm:finH}
    Let $k\in \mathbb{N}$ with $k\ge2$, $0<\alpha\leq1$, and $f\in \mathcal{C}^{k,\alpha}(\mathbb{S}^2)$. 
    Then, for any expanding sequence $\{\Omega_h\}_h$ of bounded sets which exhausts $\mathbb{Z}\times\NN$ and all $h \in \mathbb{N}$ with $\bm0 \in \Omega_h$, it holds that
    \begin{equation}\label{eq:finHbd0}
      \lVert f  - S_{\Omega_h} f \rVert_{\mathcal{C}(\mathbb{S}^2)} \leq 2^{\frac{k+\alpha}{2}-1} \pi^\alpha\, (k+3)!\, \norm{f}_{\mathcal{C}^{k,\alpha}(\mathbb{S}^2)} 
      \Big( 4 \zeta(k+\alpha-1) - \sum_{\bm{n} \in \tilde\Omega_h\setminus \{\bm0\}} \frac{1}{\lvert \bm{n} \rvert^{k+\alpha}}\Big),
    \end{equation}
    where $\zeta$ denotes the Riemann zeta function
    $
    \zeta(r)=\sum_{n=1}^\infty n^{-r}
    ,\ r>1.
    $
    In particular, we have $\lVert f  - S_{\Omega_h} f \rVert_{\mathcal{C}(\mathbb{S}^2)}\to0$ as $h\to\infty$.
\end{thm}

\begin{proof} 
    Let $h \in \mathbb{N}$ with $\bm0\in\Omega_h$.
    By \Cref{thm:hoelScon}, we have
    \begin{equation*}
        \lVert f  - S_{\Omega_h} f \rVert_{\mathcal{C}(\mathbb{S}^2)} 
        \leq \sum_{\bm{n} \in \mathbb{Z}^2\setminus \tilde\Omega_h} \big\lvert c_{\bm{n}}(\tilde{f}) \big\rvert.
    \end{equation*}
    By \Cref{thm:bdC2}, it holds 
    that
    \begin{equation}\label{eq:finHbd1}
      \lVert f  - S_{\Omega_h} f\rVert_{\mathcal{C}(\mathbb{S}^2)} 
      \leq 2^{\frac{k+\alpha}{2}-1} \pi^\alpha\, (k+3)!\, \norm{f}_{\mathcal{C}^{k,\alpha}(\mathbb{S}^2)} \sum_{\bm{n} \in \mathbb{Z}^2\setminus \Omega_h}\frac{1}{\lvert \bm{n}\rvert ^{k+\alpha}} 
      .
    \end{equation}
    In the rest of the proof, we show that
    \begin{equation}\label{eq:finHbd2}
      \sum_{\bm{n}\in \mathbb{Z}^2\setminus\{\bm0\}} \frac{1}{\lvert \bm{n} \rvert ^{k+\alpha}} = 4 \zeta(k+\alpha-1).
    \end{equation}
    It is proven in \cite[p.~308]{tornheim} that 
    \begin{equation}\label{eq:finHbd3}
      \sum_{(n,m) \in \mathbb{N}^2}\frac{1}{(n+m)^r}
      =\zeta(r-1)-\zeta(r),
      \qquad r>2.
    \end{equation}
    We split up the sum on the left hand side of \eqref{eq:finHbd2}
    to the four quadrants and the coordinate axes of $\Z^2$
    and obtain
    \begin{equation*}
      \sum_{\bm n\in \mathbb{Z}^2\setminus \{\bm0\}} \frac{1}{\lvert \bm n\rvert^{k+\alpha}}
      = 4 \sum_{(n_1,n_2)\in \mathbb{N}^2} \frac{1}{( n_1 +n_2)^{k+\alpha}}+4\sum_{n=1}^\infty \frac{1}{n^{k+\alpha}}.
    \end{equation*}
    By \eqref{eq:finHbd3}, we have
    \begin{equation*}
      \sum_{\bm n\in \mathbb{Z}^2\setminus \{\bm0\}} \frac{1}{\lvert \bm n\rvert^{k+\alpha}}
      =  4\big(\zeta(k+\alpha-1)-\zeta(k+\alpha)\big)+4\zeta(k+\alpha)
      =  4\zeta(k+\alpha-1),
    \end{equation*}
    which shows \eqref{eq:finHbd2}
    and thus finishes the proof.
\end{proof}

For the next theorem, we restrict ourselves to rectangular and {circular partial Fourier sums} \cite[p.~7~f.]{comHarAna} to obtain a bound on the speed of convergence.

\begin{thm} \label{thm:hoelSrate}
  Let $k\in \mathbb{N}$, $0<\alpha\leq1$, and $f\in\mathcal{C}^{k,\alpha}(\mathbb{S}^2)$. 
  For $h\in\mathbb{N}$, we define the circular partial DFS Fourier sum $K_h f := S_{\Omega_h}f$ associated with ${\Omega_h=\{\bm{n} \in \mathbb{Z}\times\NN\mid \lvert \bm{n}\rvert \leq h\}}$.
  Then there is a constant $M_{k,\alpha}$ depending only on $k$ and $\alpha$ such that
  \begin{equation*}
    \lVert f  - K_h f \rVert_{\mathcal{C}(\mathbb{S}^2)} 
    \leq M_{k,\alpha}\, \norm{f}_{\mathcal{C}^{k,\alpha}(\mathbb{S}^2)}\, h^{1-k-\alpha}.
  \end{equation*}
\end{thm}
\begin{proof}
  By \Cref{thm:F,thm:hoelScon}, we have
  \begin{equation*}
    \lVert f  - K_h f \rVert_{\mathcal{C}(\mathbb{S}^2)} 
    \leq \sum_{\bm{n} \in \mathbb{Z}^2,\ \lvert \bm{n}\rvert >h} 
    \big\lvert c_{\bm{n}}(\tilde{f}) \big\rvert.
  \end{equation*}
  In \cite[p.\ 184]{claFouAna}, it was shown that there exists a constant $M'_k$ such that for all $g\in \mathcal{C}^{k,\alpha}(\T^2)$ and $\ell\in\N$, we have
  \begin{equation*}
    \sum_{\substack{\bm{n}\in\Z^2,\, 2^\ell\le \lvert \bm{n}\rvert <2^{\ell+1}}} \abs{c_{\bm{n}}(g)}
    \le M'_k\, 2^{\ell(1-k)}\, 2^{-(\ell+3) \alpha}\, \abs{g}_{\mathcal{C}^{k,\alpha}(\T^2)}.
  \end{equation*}
  Then we have
  \begin{equation*}
    \sum_{\substack{\bm{n}\in\Z^2,\, \lvert\bm{n}\rvert >h}} \abs{c_{\bm{n}}(g)}
    \le M'_k\, 2^{-3\alpha}\, \abs{g}_{\mathcal{C}^{k,\alpha}(\T^2)}
    \sum_{\ell = \lfloor \log_2 h \rfloor}^\infty
    2^{\ell(1-k-\alpha)},
  \end{equation*}
  where $\lfloor \log_2 h \rfloor$ denotes the largest integer $m$ such that $2^m\le h$.
  We can evaluate the geometric sum
  $$
  \sum_{l = \lfloor \log_2 h \rfloor}^\infty
  2^{\ell(1-k-\alpha)}
  = \frac{2^{\lfloor \log_2 h \rfloor (1-k-\alpha)}}{1-2^{1-k-\alpha}}
  \le \frac{2^{k+\alpha-1} h^{1-k-\alpha}}{1-2^{1-k-\alpha}}.
  $$
  In the proof of \Cref{thm:bdC2}, we have seen that $\lvert\tilde f\rvert_{\mathcal{C}^{k,\alpha}(\T^2)} \le (k+3)!\, \norm{f}_{\mathcal{C}^{k,\alpha}(\S^2)}$.
  Hence, with $g=\tilde f$, we have
  \begin{equation*}
    \sum_{\substack{\bm{n}\in\Z^2,\, \lvert\bm{n}\rvert >h}} \abs{c_{\bm{n}}(g)}
    \le M'_k\, 2^{k-2\alpha-1}\, \frac{h^{1-k-\alpha}}{1-2^{1-k-\alpha}}\,
    (k+3)!\, \norm{f}_{\mathcal{C}^{k,\alpha}(\S^2)}.\qedhere
  \end{equation*}
\end{proof}

\begin{rem} \label{rem:rect}
  The previous theorem still holds when we replace $K_h$ by the rectangular partial Fourier sum $R_h f := S_{\Omega_h}f$ associated with $\Omega_h=\{\bm{n}\in\mathbb{Z}\times\NN\mid \lvert n_1\rvert \leq h,\, \lvert n_2\rvert \leq h\}$.
\end{rem}

\subsection{Numerical computation}
\label{sec:numerics}

In this section, we aim to verify our findings numerically.
Let $x_+$ denote the positive part of $x\in\R$.
For $\nu\in\N$ and $a\in(0,1)$, we consider the test function
$$
f_\nu(\bm\xi)
:= ((\xi_3-a)_+)^{\nu+1}
,\qquad
\bm{\xi}\in\S^2,
$$
which has an extension $f_\nu^*$ defined by the same formula for $\bm{\xi}\in\R^3$.
Obviously, we have $f_\nu\in\mathcal{C}^\nu(\S^2)$.
The derivatives of $f^*$ with respect to the first two components vanish, and we see that ${D^{\nu\bm{e}^3}(\bm{\xi}) = \nu!\, (\xi_3-a)_+}$ is Lipschitz-continuous. Hence,
we have $f_\nu\in \mathcal{C^{\nu,\alpha}}(\S^2)$ for $0<\alpha<1$, cf.\ \Cref{prop:lipH}.
We test the DFS method with a linear combination of rotated versions of $f_3$, see \Cref{fig:f1}.

\begin{figure}[ht]\centering
  \includegraphics[height=.32\textwidth]{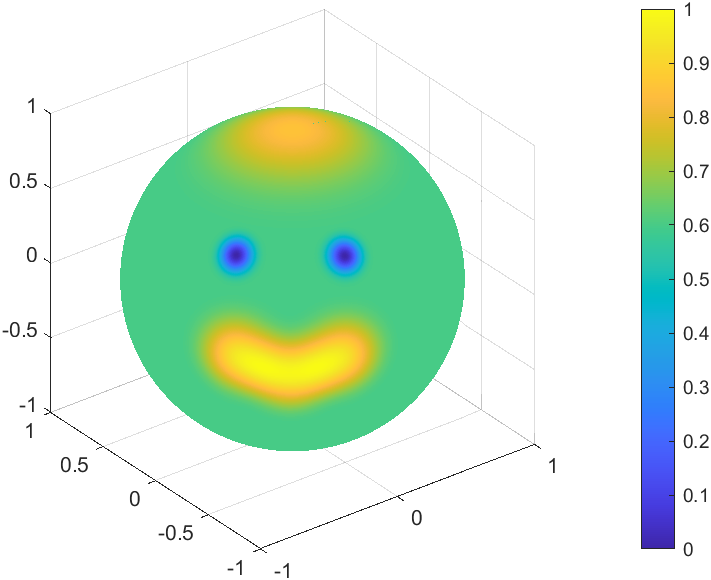}\qquad
  \includegraphics[height=.32\textwidth]{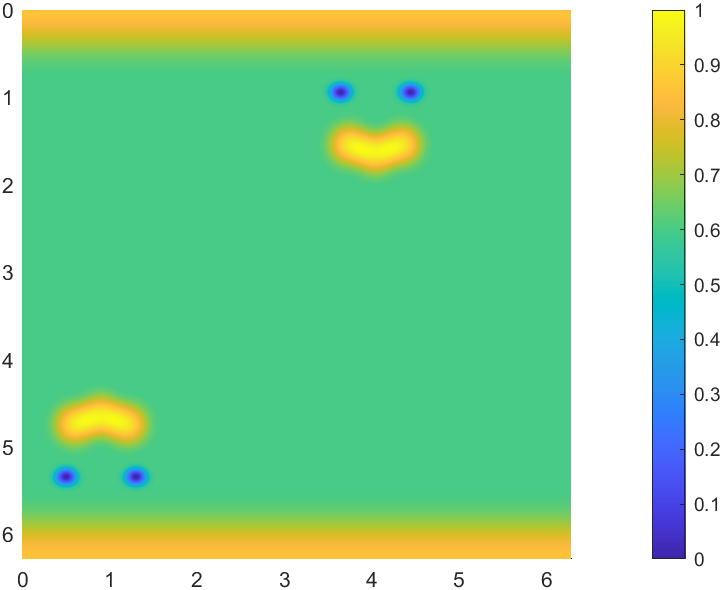}
  \caption{Test function $f_3$ (left) on the sphere $\S^2$ and its DFS function $\tilde f_3$ (right) on the torus $\T^2$.
    \label{fig:f1}
  }
\end{figure}

We compute the rectangular Fourier sum $R_h$ defined in \Cref{rem:rect} for different degrees~$h$
on a uniform grid of size $2400\times1200$ in the spherical coordinates $(\lambda,\theta)$
by first computing the Fourier sum \eqref{eq:DFSseries1} of $\tilde f$ on the "doubled" grid of $2400\times2400$ points on $\T^2$, and then transfer this to the sphere $\S^2$ by \Cref{thm:F}.
The Fourier coefficients $c_k(\tilde f)$ were computed approximately by an FFT of $\tilde f$ on a finer grid of $4800\times4800$ points.

For comparison, we also compute the spherical harmonics expansion \eqref{eq:sh-series}, truncated to the degree $n\leq h$.
The approximation error is similar to the DFS Fourier series, see \Cref{fig:error}.
Note that $R_h$ consists of $(h+1)(2h+1)$ summands, which are about twice as many as the spherical harmonics expansion with $(h+1)^2$ summands.
However, the computation time for the expansion of degree $h=1024$ is about 0.06\,s for the DFS Fourier expansion and 0.98\,s for the spherical harmonic expansion with the algorithm \cite{nfft3} on a standard PC with Intel Core i7-10700.

\begin{figure}[ht]\centering
    \begin{tikzpicture}
    \begin{loglogaxis}[xlabel=$h$, ylabel={max. error}, width=.54\textwidth, height=0.32\textwidth, xmin=6, xmax=1200, ymin=4e-9, ymax=1, axis x line*=bottom, axis y line*=left, legend pos=outer north east, legend style={cells={anchor=west}}]
      \addplot+[thick,mark=none,blue] table[x index=0,y index=1] {matlab/f1_error.dat};
      \addplot+[thick,mark=none,purple,dashed] table[x index=0,y index=2] {matlab/f1_error.dat};
    \end{loglogaxis}
  \end{tikzpicture}
  \caption{Logarithmic plot of the maximal error $\norm{f-R_hf}_{\mathcal{C}(\S^2)}$ of the DFS Fourier sum (blue) and the maximal error of the truncated spherical harmonics series (red, dashed), depending on the degree $h$.
  While the error behavior is similar, the DFS Fourier expansion can be computed much faster.
  \label{fig:error}
}
\end{figure}
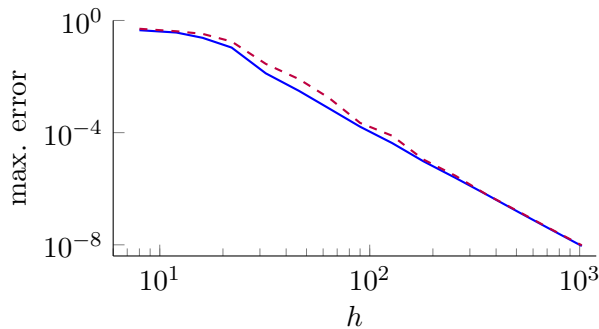

\section{Sobolev spaces}\label{sec:sobolev}

Sobolev spaces on the torus play an important role in harmonic analysis. 
A question which naturally arises is whether the DFS method preserves Sobolev spaces like it preserves differentiability and Hölder classes.
We will see that this question has to be answered in the negative; the DFS method does not map spherical Sobolev spaces to Sobolev spaces of the same order on the torus. 

Sobolev spaces on $\mathbb{R}^d$ can be defined via weak derivatives, see, e.g., \cite[p.~60]{sobolev}. 
We adapt this notion to $\T^d$ in the following.
  Let us denote by
  $
  \mathcal{C}^\infty_{\mathrm{c}}(\R^d)
  $
  the space of smooth functions with compact support.
  For $\bm{\beta} \in \NN^d$ and locally integrable functions $f,g \colon \R^d\to\C$, we say that $g$ is the $\bm{\beta}$-\emph{weak derivative of $f$} and write $D^{\bm{\beta}}f=g$ if
  \begin{equation*}
      \int_{U} f(\bm{x})\, D^{\bm{\beta}}u(\bm{x}) \d\bm{x}
      =(-1)^{\lvert \bm{\beta}\rvert} \int_U g(\bm{x})\, u(\bm{x}) \d \bm{x} 
      \quad \text{for all } u\in \mathcal{C}^\infty_{\mathrm{c}}(\R^d).
  \end{equation*}

\begin{defin}\label{defin:sob}
    We define the \emph{Sobolev space} $H^k(\mathbb{T}^2)$ of order $k \in \NN$ on the torus as the set of all functions $f \in L_2(\mathbb{T}^2)$ such that $D^{\bm{\beta}}f\in L_2(\mathbb{T}^2)$ for all $\bm{\beta}\in B^2_k$.
    This is a Hilbert space equipped with the norm
    \begin{equation*}
        \lVert f \rVert_{H^k(\mathbb{T}^2)}:=
            \left( \sum_{\bm{\beta}\in B^2_k} \lVert D^{\bm{\beta}}f \rVert_{L_2(\mathbb{T}^2)}^2 \right)^{\frac{1}{2}}.
    \end{equation*}
\end{defin}

  We define the radial extension $f^\star$ of $f \colon\mathbb{S}^{2} \to \mathbb{C}$ by
  \begin{equation*}
      f^\star \colon \mathbb{R}^3\setminus \{0\} \to \mathbb{C}, \ f^\star(\bm{x}) = f\left(\frac{\bm{x}}{\lVert \bm{x} \rVert}\right).
  \end{equation*}
  If $f^\star$ is differentiable, the {surface gradient} of $f$ is given by
  \cite[p.~78]{mueller}
  \begin{equation*}
      \nabla^\star f\colon\mathbb{S}^{2}\to \mathbb{C}^3, \ \nabla^\star f(\bm{x}) = \nabla f^\star(\bm{x}).
  \end{equation*}

There are many equivalent definitions of spherical Sobolev spaces, e.g., with spherical harmonics, see \cite[sec.~6.2]{michel}. 
For our purposes, the following equivalent characterization derived in \cite[p.~17]{quellmalz} is convenient.
\begin{defin}\label{defin:sobs}
    We set the zeroth order spherical Sobolev space as $H^0(\S^2) := L_2(\S^2)$.
    For $k \in \mathbb{N}$, the spherical \emph{Sobolev norm} of $f \in \mathcal{C}^k(\mathbb{S}^2)$ is defined recursively by
    \begin{equation*}
        \lVert f \rVert_{H^k(\mathbb{S}^2)}
        :=\left( \sum_{i=1}^3 \left\lVert (\nabla^\star f)_i \right\rVert_{H^{k-1}(\mathbb{S}^2)}^2
        +\frac{1}{4}\lVert f \rVert^2 _{H^{k-1}(\mathbb{S}^2)}\right)^{\frac{1}{2}}.
    \end{equation*}
    The \emph{Sobolev space} $H^k(\mathbb{S}^2)$ of order $k$ is
    the closure of the subset of functions $f\in\mathcal{C}^k(\mathbb{S}^2)$ of finite Sobolev norm $\lVert f \rVert_{H^k(\mathbb{S}^2)}$ with respect to this norm. 
\end{defin}

A function which is in $H^1(\mathbb{R}^2)$ but unbounded near the origin can be found in \cite[(4.43)]{sobolev}. We adapt this example to our spherical setting, yielding a function in $H^1(\mathbb{S}^2)$ which is unbounded around the poles and whose DFS function is not in $H^1(\mathbb{T}^2)$. 
The intuition behind this result is that the DFS transform maps a small area around the poles to an enlarged area on the torus. Therefore, the spherical Sobolev norm does not sufficiently control the behavior around the poles to ensure a finite integral on the torus. 
Due to the Sobolev embedding theorem, all functions in a spherical Sobolev space of order greater than one are bounded and continuous. 

\begin{thm} \label{thm:count}
    Let
    \begin{equation*}
        f \colon\mathbb{S}^2 \to \mathbb{C}, \ f(\bm{\xi}) = \begin{cases}
        \ln\left(\ln\frac{8}{\sqrt{1-\xi_3^2}}\right), &\ \lvert \xi_3\rvert \not=1,\\
        0, & \text{ otherwise}.\end{cases}
    \end{equation*}
    Then $f \in H^1(\mathbb{S}^2)$ and $\tilde{f} \not\in H^1(\mathbb{T}^2)$.
\end{thm}

\begin{proof}
  For $n \in \mathbb{N}$, we set
  \begin{equation*}
    f_n \colon\mathbb{S}^2 \to \mathbb{C}, \ f_n(\bm{\xi}) = \ln\left(\ln\frac{8}{\sqrt{1-\xi_3^2}+\frac{1}{n}}\right),
  \end{equation*}
  which converge pointwise almost everywhere, more precisely on $\{\bm{\xi} \in \S^2\mid \lvert \xi_3\rvert \not=1\}$, to $f$ for $n\to\infty$.
  We will show that $f_n \to f$ in $H^1(\mathbb{S}^2)$ for $n\to\infty$.
    Let $n \in \mathbb{N}$.
    Clearly, the function $f_n$ is non-negative and continuously differentiable. Let $(\lambda,\theta)\in [-\pi,\pi]\times[0,\pi]$, then $0\leq \sin\theta =\sqrt{1-(\cos\theta)^2}$ and hence
    \begin{equation*}
        (f_n \circ \phi)(\lambda,\theta) 
        =\ln\left(\ln\frac{8}{\sin\theta+\frac{1}{n}}\right).
    \end{equation*}
	  By \cite{sphericalgrad}, the surface gradient of the differentiable function $f_n$ is given by
		\begin{equation*}
		\nabla^\star f_n(\phi(\lambda,\theta))
    =\begin{pmatrix}\cos(\lambda)\cos(\theta)\\\sin(\lambda)\cos(\theta)\\-\sin(\theta)\end{pmatrix} \frac{\partial (f_n\circ \phi)}{\partial \theta}(\lambda,\theta)
    +\begin{pmatrix}-\sin\lambda\\ \ \cos\lambda\\0\end{pmatrix} \frac{1}{\sin\theta} \frac{\partial (f_n\circ\phi)}{\partial \lambda}(\lambda,\theta).
		\end{equation*}
    Hence, we have
    \begin{equation} \label{eq:grad-fn}
        (\nabla^\star f_n\circ \phi)(\lambda,\theta)
        =- \begin{pmatrix}\cos(\lambda)\cos(\theta)\\\sin(\lambda)\cos(\theta)\\-\sin(\theta)\end{pmatrix}  \frac{\cos\theta}{(\sin\theta+\frac{1}{n})\, \ln\frac{8}{\sin\theta+\frac{1}{n}}}.
    \end{equation}
    Let $(\lambda,\theta)\in [-\pi,\pi]\times(0,\pi)$. Clearly, we have
    $\sin(\theta)< \sin(\theta)+\frac{1}{n}$,
    which implies that
    \begin{equation*}
        \lvert (f_n\circ\phi)(\lambda,\theta)\rvert \leq \ln\left(\ln\frac{8}{\sin\theta}\right)
        = (f\circ\phi)(\lambda,\theta),
    \end{equation*}
    since the logarithm is increasing. Furthermore, since the function $x\mapsto x \ln\frac{8}{x}$ is increasing on the interval $(0,\frac{8}{\e})$, we have  that
    \begin{equation*}
        \sin(\theta)\, \ln\left(\frac{8}{\sin\theta}\right)\leq \left(\sin\theta+\frac{1}{n}\right) \,\ln\left(\frac{8}{\sin\theta+\frac{1}{n}}\right).
    \end{equation*}
    Hence, we have by \eqref{eq:grad-fn} for all $i\in[3]$
    \begin{equation*}
        \lvert (\nabla^\star f_n \circ \phi)_i(\lambda,\theta)\rvert
        \leq \frac{\lvert \cos\theta\rvert}{\sin(\theta)\, \ln\frac{8}{\sin\theta}}.
    \end{equation*}
    By l'Hôpital's rule, we have 
    $\lim_{\theta \downarrow 0}\left(\ln\left(\ln\frac{8}{\sin\theta}\right)\right)^2\,\sin\theta=0$ and hence, by \eqref{eq:intS},
    \begin{equation*}
      \int_{\S^2} \abs{f(\bm\xi)}^2\d\bm\xi
      = 2\pi \int_{0}^\pi \left(\ln\left(\ln\frac{8}{\sin\theta}\right)\right)^2 \, \sin(\theta) \,\mathrm{d}\theta<\infty.
    \end{equation*}
    This implies $f \in L_2(\mathbb{S}^2)$.
    By Lebesgue's dominated convergence theorem applied to the function $\abs{f_n-f}$, we have $f_n \to f$ in $L_2(\mathbb{S}^2)$ as $n \to \infty$. 
    Furthermore, by substituting $t= \ln\frac{8}{\sin\theta}$, we have
    \begin{equation*}
        \int_{0}^\pi \frac{(\cos\theta)^2}{(\sin\theta)^2\left(\ln\frac{8}{\sin\theta}\right)^2}\sin(\theta) \,\mathrm{d}\theta
        \leq 2 \int_0^{\frac{\pi}{2}}\frac{\cos\theta}{\sin\theta}\left(\ln\frac{8}{\sin\theta}\right)^{-2}\,\mathrm{d}\theta
        =2\int_{\ln(8)}^\infty \frac{1}{t^2}\,\mathrm{d}t<\infty.
    \end{equation*}
    Therefore, again by the dominated convergence theorem, we have $f_n \in H^1(\mathbb{S}^2)$ for all $n \in \mathbb{N}$ and that $\{f_n\}_n$ is a Cauchy sequence in $H^1(\mathbb{S}^2)$. By completeness, it follows that there exists a limit in $H^1(\mathbb{S}^2)$ of $\{f_n\}_n$. We can identify this limit as $f$ since $H^1(\mathbb{S}^2)$ convergence clearly implies $L_2(\mathbb{S}^2)$ convergence. We conclude that $f \in H^1(\mathbb{S}^2)$.
    
    To show that $\tilde{f} \not\in H^1(\mathbb{T}^2)$, we assume that $\tilde{f}$ has a weak derivative $D^{\bm{e}_2}\tilde{f} \in L_1^\mathrm{loc}(\mathbb{R}^2)$. Since $\tilde{f}$ is classically differentiable on $(\pi,\pi)\times (0,\pi)$, the fundamental lemma of calculus of variations shows that
    \begin{equation*}
        D^{\bm{e}_2}\tilde{f}(\lambda,\theta)
        =\frac{-\cos\theta}{\sin(\theta)\, \ln\frac{8}{\sin \theta}} 
        \quad\text{almost everywhere on }(\pi,\pi)\times(0,\pi).
    \end{equation*}
    However, we have
    \begin{equation*}
        \big\lVert \tilde{f} \big\rVert_{H^1(\mathbb{T}^2)}^2
        \geq \int_{(-\pi,\pi)\times(0,\pi)} \big\lvert D^{\bm{e}_2}\tilde{f}(\bm{x})\big\rvert^2\,\mathrm{d}\bm{x}
        =2 \pi \int_0^\pi \frac{(\cos\theta)^2}{(\sin\theta)^2 \left(\ln\frac{8}{\sin\theta}\right)^2}\,\mathrm{d}\theta =\infty.
    \end{equation*}
    Hence, $\tilde{f} \not\in H^1(\mathbb{T}^2)$.
\end{proof}

\section*{Acknowledgments}
The authors thank Gabriele Steidl for making valuable comments to improve this article.
The second author thanks Tino Ullrich for insightful discussions about Hölder spaces.

\printbibliography[heading=bibintoc]

\end{document}